\documentclass[12pt]{article}
\usepackage[utf8x]{inputenc}
\usepackage[english]{babel}
\usepackage{amsmath,amsfonts,amssymb,amsthm,lmodern}
\usepackage{hyperref}
\usepackage{tikz}
\usepackage[margin=2.5cm]{geometry}
\title{Factorizations of Cyclic Groups and Bayonet Codes}
\author{Christophe Cordero\footnote{ccordero@unisa.it}}
\date{
{\small$^*$Dipartimento di Informatica ed Applicazioni, Universit\`a di Salerno,\\
via Giovanni Paolo II 132, Fisciano, 84084, ITALY.}\\[2ex]
Nov 20, 2022}
\newtheorem{theorem}{Theorem}[section]
\newtheorem{proposition}[theorem]{Proposition}
\newtheorem{lemma}[theorem]{Lemma}
\newtheorem{definition}[theorem]{Definition}
\newtheorem{example}[theorem]{Example}
\newtheorem{remark}[theorem]{Remark}
\newtheorem{corollary}[theorem]{Corollary}
\newtheorem{conjecture}[theorem]{Conjecture}
\newcommand{\DEF}[1]{\textit{\textbf{#1}}}
\newcommand{\EnsCal}[1]{\mathcal{#1}}
\newcommand{\Alpha}{\EnsCal{A}}
\newcommand{\B}{\EnsCal{B}}
\newcommand{\Zn}[1]{\mathbb{Z}_{#1}}
\newcommand{\Z}{\Zn{n}}
\newcommand{\ENGENDRE}[1]{\left\langle #1 \right\rangle}
\newcommand{\PAR}[1]{\left(#1\right)}
\newcommand{\ENS}[1]{\left\{#1\right\}}
\newcommand{\CARD}[1]{\left|#1\right|}
\newcommand{\SOM}[1]{\underline{#1}}
\newcommand{\Modn}[1]{\equiv_{#1}}
\newcommand{\Mod}{\Modn{n}}
\newcommand{\MODn}[2]{\overline{#1}^{#2}}
\newcommand{\MOD}[1]{\MODn{#1}{n}}
\newcommand{\dans}{\in_{n}}
\newcommand{\ENTIERS}[1]{\left[ #1 \right]}
\newcommand{\entiers}{\ENTIERS{n}}
\renewcommand{\L}[3]{L_{#1}^{#2}\PAR{#3}}
\newcommand{\R}[3]{R_{#1}^{#2}\PAR{#3}}
\newcommand{\DUAL}[1]{\delta\PAR{#1}}

\begin{document}
\maketitle

	\begin{abstract}    	
    We study the (variable-length) codes of the 
    form~$X \cup \ENS{a^n}$ where~$X \subseteq a^*\omega a^*$ and~$\CARD{X} = n$. 
    We extend various notions and results from \textit{factorizations of cyclic groups} theory 
    to this type of codes.
    
    In particular, 
    when~$n$ is the product of at most three primes or has the form~$p^k q$
    (with~$p$ and~$q$ prime), we prove that they are composed of prefix and suffix codes. 
    We provide counterexamples for other~$n$.
    It implies that the long-standing \textit{triangle conjecture} 
    is true for this type of~$n$.
    We also prove a conjecture about the size of a potential counterexample to the conjecture.
	\end{abstract}

	\section*{Introduction}

	Schützenberger founded and developed the theory of (variable-length) \textit{codes} in the~1960s 
	in order to study encoding problems raised by Shannon's information theory. 
	Since then, the theory has undergone its own development and links with monoids, automata, 
	combinatorics on words and factorizations of cyclic groups have emerged. 
	We refer the reader to the book~\cite{berstel2010codes} for an introduction to the theory of codes.
	
	Let~$\Alpha$ be an alphabet containing letters~$a$ and~$b$.
	A \DEF{code} is a subset~$X \subseteq \Alpha^*$ such that for 
	all~$t, t' \geq 0$ and~$x_1, \dots, x_{t}, y_1, \dots, y_{t'} \in X$ the condition
	\begin{equation*}
    	x_1 \, x_2 \, \cdots \, x_{t} \, = \, y_1 \, y_2 \, \cdots \, y_{t'}
	\end{equation*}
	implies~$t = t'$ and~$x_i = y_i$, for~$i = 1, \dots, t$.
	For example, the set~$\ENS{aabb, abaaa, b, ba}$ is not a code because
	\begin{equation*}
		(b) (abaaa) (b) (b) = (ba) (ba) (aabb).
	\end{equation*}
	A straight forward way to create a code is to build a \textit{prefix set} 
	(respectively \textit{suffix set}),
	which is a set of words such that none of its words is a prefix
	(resp. suffix) of another one.
	For example, the set
	\begin{equation} \label{intro:ex_code_prefixe}
		\ENS{aaa, ab, aab, bba, ba}
	\end{equation}
	is prefix but not suffix.
	According to Proposition~2.1.9 from~\cite{berstel2010codes}, a prefix (resp. suffix) set 
	different than~$\ENS{\varepsilon}$ is a code,
	where~$\varepsilon$ is the empty word. 
	Such codes are called \DEF{prefix codes} (resp. \DEF{suffix codes}).
	So the set~\eqref{intro:ex_code_prefixe} is a prefix code.
	
	Since any code is included in a \DEF{maximal} code 
	(code that is not included in another code), 
	most of the theory of codes 
	is dedicated to the study of finite maximal codes.
	
	One of the main conjecture about the characterization of finite maximal codes is
	the \DEF{factorization conjecture} from 
	 Schützenberger~\cite{schutzenberger1965codes}.
	This conjecture states that for any finite maximal code~$M$, 
	there exists finite sets~$P, S \subseteq \Alpha^*$ such that
	\begin{equation} \label{conjecture_factorisation}
		\SOM{M} - 1 = \SOM{P} \PAR{\SOM{\Alpha} - 1} \SOM{S},
	\end{equation}
	where given a set~$X \subseteq \Alpha^*$, we denote its formal sum
	\begin{equation*}
		\sum\limits_{x \in X} x
	\end{equation*}
	by~$\SOM{X}$.
	The best known result about the conjecture is from 
	Reutenaeur~\cite{reutenauer1985noncommutative}. He proved 
	that for any finite maximal code~$M$, there 
	exists polynomials~$\SOM{P}, \SOM{S} \in \mathbb{Z}\langle\langle \Alpha \rangle\rangle$
	(the set of formal power series 
	of~$\Alpha^*$ over~$\mathbb{Z}$)
	such that~\eqref{conjecture_factorisation}.
	
	During some unsuccessful attempts to prove the conjecture, 
	Perrin and Schützenberger proposed an intermediate conjecture called 
	the \textit{triangle conjecture}~\cite{perrin1977probleme}. 
	It is stated as follows: for any \DEF{bayonet code}~$X$, i.e. for any 
	code~$X \subseteq a^* ba^* $, we have 
	\begin{equation} \label{intro:triangle}
		\CARD{ \ENS{x \text{ such that } x \in X \text{ and } |x| \leq k} } 
		\leq k, \text{ for all } k \geq 0.
	\end{equation}
	However, Shor found a counterexample~\cite{shor1985counterexample}.

	Since then, variants of the triangle conjecture have been proposed. In particular 
	the one that we nowadays call, by an abuse of language, the \DEF{triangle conjecture}
	 which suggests that any bayonet code~$X$ either satisfies the inequalities~\eqref{intro:triangle}
	  or it is not included in a finite maximal code.
	This conjecture is implied by the factorization conjecture.
	
    A stronger version of the triangle conjecture proposed by Zhang and Shum~\cite{zhang2017finite} 
	states that
	for all finite maximal codes~$M$,~$\omega \in \Alpha^*$,
	and~$k \geq 0$, we have
	\begin{equation*} 
		\CARD{ 
		\ENS{a^i \omega a^j \text{ such that } 
		a^i \omega a^j \in M^* \text{ and } i+j < k} } 
		\leq k. 
	\end{equation*}

	In this paper, our main subject is the (subsets of) codes
     concerned by the triangle conjectures, which are
     the codes of the form 
     \begin{equation*}
     	X \cup \ENS{a^n},
	\end{equation*}      
	where~$X \subseteq a^{\ENS{0,1,\dots, n-1}} ba^{\ENS{0,1,\dots, n-1}}$ and~$\CARD{X} = n$, 
	for a given~$n \geq 1$. 
	We call	them~$n$-\DEF{complete bayonet codes} (cbc). 
	"Complete" refers to the fact that such a code cannot 
	contain more elements, according to Proposition~2.1 of~\cite{de1985some}.
	We extend various notions and results from \textit{factorizations of cyclic groups} theory 
    to cbc.
    The framework we develop about cbc generalizes and simplifies the work 
	recently made about the triangle conjectures such 
	as~\cite{zhang2017finite,ResearchGateZhangShum,Felice2022}
	and allows us to improve the best known results.
    
    In particular, 
    when~$n$ is the product of at most three primes or has the form~$p^k q$
    with~$p$ and~$q$ prime (we call those numbers \DEF{cbc Haj\'{o}s numbers}), we prove 
    that any code~$X \cup \ENS{a^n}$, where~$X \subseteq a^* \omega a^*$,~$\omega \in \Alpha^*$, 
    and~$\CARD{X}=n$, 
    is a composition of prefix and suffix codes.
    We provide counterexamples in other cases.
	This implies that the Zhang and Shum conjecture and therefore
	the triangle conjecture is true for cbc Haj\'{o}s numbers.
	Moreover, our Theorem~\ref{th:d_conjecture} proves a conjecture about the size of a potential
	counterexample to the triangle conjectures
	and our Theorem~\ref{Theorem_carac_periodic} gives a structural property of 
	codes satisfying the factorization conjecture.  
		
	Our paper is organized as follows.
	In the first section, we mostly introduce and recall some concepts 
	that relate cbc to finite maximal codes.	
	In the second section, we study some operations on cbc
	that, among others, lead to a criterion that a code must satisfies in order to be
	 included in a finite maximal code.
    In the third section, we show that each cbc can be associated 
	to a notion that we call \textit{border}. 
	This notion exhibits a link between \textit{factorizations of cyclic groups} 
	and cbc. We deduce from it a characteristic about finite maximal codes.	 
	Then we show various operations to build \textit{borders} from others. 
	In the fourth section, similarly to factorization theory, we introduce a \textit{periodic} 
	and a \textit{Haj\'{o}s} notion 
	for cbc. Then we show that \textit{Haj\'{o}s cbc} are cbc bordered by a 
	\textit{Krasner factorization} which, in particular, provides a structural property of codes 
	that
	satisfy the factorization conjecture. 	
	In the fifth section, we prove that any cbc of size~$n$,
	where~$n$ is a cbc Haj\'{o}s numbers, is of Haj\'{o}s.
	We provide counterexamples in other cases.
	In the sixth section, we show that Haj\'{o}s cbc are composed of prefix and suffix codes 
	and thus
	are included in some finite maximal codes. We also deduce from it that the triangle 
	conjectures are true for cbc Haj\'{o}s numbers.	
	In section seven, thanks to the framework on borders, 
	we prove a conjecture about the size of a potential
	counterexample to the triangle conjecture.
	Finally, we conclude by exposing our main perspectives.

	\section{Complete Bayonet Code}

	It is well known\footnote{see, for example, Proposition~2.5.15 
   	of~\cite{berstel2010codes}.} 
   	that for any finite maximal code~$M$, there exists a unique integer~$n$ 
   	such that~${a^n \in M}$, it is called the \DEF{order} of the letter~$a$.
   	Most of the known characterizations of finite maximal codes are based 
   	on the order of one of its letter.
   	Such as Restivo, Salemi, and Sportelli who have shown~\cite{restivo1989completing} 
   	the following link between \textit{factorizations} (\textit{of cyclic groups}) and 
   	finite maximal codes.
   	
	\begin{theorem} \label{Preliminaire:th:restivo}
	If~$M$ is a finite maximal code such that~$b, a^n \in M$ then the 
	ordered pair
    \begin{equation*}
       \PAR{ \ENS{ k, \, a^k b^+ \in M },
        	  \ENS{ k, \, b^+ a^k \in M} }
    \end{equation*} 
	is a factorization of size~$n$. 
	\end{theorem}
	
	We recall some basics notions about \textit{factorizations}.
	However, we refer the reader to the books~\cite{szabo2004topics, szabo2009factoring} 
	in order to find proofs of those recalls and	
	for an introduction to the more general theory of factorizations of abelian groups.	
	Given~$P,Q \subseteq \Zn{}$,
	an ordered pair~$\PAR{P,Q}$ is a \DEF{factorization} of size~$n$ 
	if and only if for all~$k \in \entiers$ ($\entiers$ denotes the set~$\ENS{0, 1, \dots, n-1}$),
	there exists a unique~$(p,q) \in P \times Q$ such that~$k = p +q$ in~$\Z$ 
	(the cyclic group of size~$n$).
	In particular,~$n$ must be equal to~$\CARD{P} \times \CARD{Q}$.
	For example, the ordered pair
	\begin{equation} \label{exemple:factorisation}
		\PAR{ \ENS{ 4, 5, 6, 7 } , \ENS{ 1, 5 }}
	\end{equation}
	is a factorization of size~$8$.
	A factorization~$(P,Q)$ is called \DEF{periodic} with period~$m \neq 0$ (in~$\Z$), 
	if~$P$ or~$Q$ is~$m$-periodic 
	in~$\Z$, i.e. 
	if
	\begin{equation*}
		m + P = P \text{ or } 
		m + Q = Q \text{ in } \Z.
	\end{equation*}
	 For example, the factorization~\eqref{exemple:factorisation} is~$4$-periodic 
	because~$4 + \ENS{1, 5} = \ENS{1, 5}$ in~$\mathbb{Z}_8$.

	A factorization~$\PAR{P,Q}$ is said to be \DEF{normalized} if~$0 \in P$
	and~$0 \in Q$. If~$\PAR{P,Q}$ is a ($m$-periodic) factorization 
	then for any~$p \in P$ and~$q \in Q$,
	\begin{equation*}
		\PAR{P-p,Q-q}
	\end{equation*}		
	is a ($m$-periodic) 
	normalized factorization.
	
	We recall that given a factorization~$\PAR{P,Q}$ of size~$n$, if~$P$ is periodic 
	then the set made of its periods and~$0$, i.e.
	\begin{equation*}
		\ENS{m, m+P =P \text{ in } \Z},
	\end{equation*}		
	is a subgroup of~$\Z$. 
	Moreover, if~$P$ is~$m$-periodic then~$\CARD{Q}$ divides~$m$. 	
	A set~$U$ is~$m$-periodic in~$\Z$ if and only if 
	\begin{equation*}
		U = \MODn{U}{m} + m\ENTIERS{\frac{n}{m}} \hspace{10pt}\text{ in } \Z,
	\end{equation*}
	where~$\MODn{U}{m}$ denotes the set~$\ENS{\MODn{u}{m}, u \in U}$ and where~$\MODn{u}{m}$ 
	denotes the remainder of the Euclidean division 
	of~$u$ by~$m$.
	
	A factorization~$(P,Q)$ of size~$n$ is 	
	said to be of \DEF{Haj\'{o}s} if and only if~$n = 1$ or if~$P$ (respectively~$Q$) is~$m$-periodic 
	and if
	\begin{equation*}
		\PAR{\MODn{P}{m},Q} \hspace{10 pt}
		\PAR{\text{resp. } \PAR{P,\MODn{Q}{m}}}
	\end{equation*}
	is again a Haj\'{o}s factorization of size~$m$. All factorizations of size~$n$
	where~$n$ is the product of at most four primes or a number of the form~$p^kq$, where~$k \geq 1$ 
    and~$p,q$ are primes, are Haj\'{o}s factorizations. Those numbers are called 
    \DEF{Haj\'{o}s numbers}. 
    Smallest non-Haj\'{o}s factorizations are therefore of sizes~$72$ and~$108$~\cite{oeis}. 
    See~\cite{de1953factorization} for such an example of non-Haj\'{o}s factorization of size~$72$.

	Thanks to Theorem~\ref{Preliminaire:th:restivo}, we can prove that the code
	\begin{equation} \label{code_restivo}
		\ENS{a^5, ab, b, baa},
	\end{equation}
	found by Restivo~\cite{restivo1977}, cannot be included in a finite maximal code 
	because there is no factorization of the form
    \begin{equation*}
        \PAR{ \ENS{ 0,1} \subseteq P,
               \ENS{ 0,2} \subseteq Q },
    \end{equation*}
    where~$\CARD{P} \times \CARD{Q} = 5$.
	
	In this paper, we expose a deeper link between the theory of codes and factorizations, 
	starting from the following notion introduced by Perrin and Schützenberger 
	 in~\cite{Perrin_Schutzenberger}.
	
	\begin{definition} Given a code~$M$ such that~$a^n \in M$ 
	and a word~$\omega \in \Alpha^*$,
	we set
	\begin{equation} \label{Preliminaire:def:eq:Cw}
		C_M(\omega) := 
		\ENS{ a^{\MOD{i}}ba^{\MOD{j}} \text{ such that }
		a^{i} \omega a^{j} \in M^* }.
	\end{equation}
	\end{definition} 
   	For example, given the finite maximal code
   	\begin{equation} \label{ex:code_max}
   		E := \ENS{b, ab, a^4, a^2ba, a^3b, a^2b^2},
   	\end{equation}
   	we have
   	\begin{equation*}
   		C_E(b) = \ENS{b, ab, a^2ba, a^3b}
   		\text{ and }
   		C_E(bb) = \ENS{b, ab, a^2b, a^3b}.
   	\end{equation*}
   	
   	As shown in Proposition~2.2 of~\cite{ResearchGateZhangShum}, sets 
   	of type~\eqref{Preliminaire:def:eq:Cw} form codes.
   	
	\begin{proposition} \label{prop:cbc} If~$M$ is a code 
	containing~$a^n$ then for any~$\omega \in \Alpha^*$, the set
		\begin{equation}\label{prop:eq:cbc}
			 \ENS{a^n} \cup C_M(\omega)
		\end{equation}
		is a code.
	\end{proposition} 
	
	Our statement is slightly different, we produce a straightforward proof.
	
	\begin{proof}
		Given a word~$\omega$ and a code~$M$ containing~$a^n$, 
		if the set~\eqref{prop:eq:cbc} 
		is not a code then there exists
		 \begin{equation*}
		 a^{i_1} \omega a^{j_1}	, \dots, a^{i_t} \omega a^{j_t},
		 a^{k_1} \omega a^{\ell_1}, \dots, a^{k_t} \omega a^{\ell_t} \in M^*	
    	\end{equation*}	
		such that~$\MOD{j_1} \neq \MOD{\ell_1}$ and 
		\begin{equation*}
		 a^{\MOD{i_1}} b a^{\MOD{j_1}} 	\dots a^{\MOD{i_t}} b a^{\MOD{j_t}}
		 \Mod
		 a^{\MOD{k_1}} b a^{\MOD{\ell_1}}	\dots a^{\MOD{k_t}} b a^{\MOD{\ell_t}},	
    	\end{equation*}	
    	where~$\Mod$ denotes the congruence over~$\mathbb{Z}\langle\langle \Alpha \rangle\rangle$ 
	 defined by the relation~$a^n = \varepsilon$.
    	Thus
		 \begin{equation*}
		 a^{i_1} \omega a^{j_1}	 \dots a^{i_t} \omega a^{j_t}
		 \Mod
		 a^{k_1} \omega a^{\ell_1} \dots a^{k_t} \omega a^{\ell_t}.	
    	\end{equation*}	
    	Which implies that~$M$ is not a code since words from the non-empty set
		\begin{equation*}
		 \PAR{a^n}^*\PAR{a^{i_1} \omega a^{j_1}}\PAR{a^n}^* 	
		 \dots \PAR{a^{i_t} \omega a^{j_t}}\PAR{a^n}^*
		 \cap
		 \PAR{a^n}^* \PAR{a^{k_1} \omega a^{\ell_1}} \PAR{a^n}^*	
		 \dots \PAR{a^{k_t} \omega a^{\ell_t}}  \PAR{a^n}^*
    	\end{equation*}		
    	can be decompose in two different ways.
    	This ends the proof by contradiction.
	\end{proof}
   	
   	Perrin and Schützenberger introduced 
   	the notion~\eqref{Preliminaire:def:eq:Cw} as a characterization 
   	of finite maximal codes.
	\begin{theorem} \label{Preliminaire:th:schutz}
	If~$M$ is a finite code such that~$a^n \in M$ then
	\begin{equation} \label{Preliminaire:th:schutz:equivalence1}
		M \text{ is maximal}
		 \hspace{8pt} \iff \hspace{8pt}		
		\forall \omega \in \Alpha^*, \hspace{10pt} 
		\CARD{C_M(\omega)} = n. 
	\end{equation}
	\end{theorem}
	
	The left to right implication of~\eqref{Preliminaire:th:schutz:equivalence1} is demonstrated 
	as Proposition~12.2.4 in~\cite{berstel2010codes} 
	and the converse is true according to Theorem~2.5.13
	of~\cite{berstel2010codes}.
	The "finite" hypothesis is necessary 
	because the code~\eqref{code_restivo} of Restivo is contained in some infinite maximal codes 
	and none of which verify~\eqref{Preliminaire:th:schutz:equivalence1} 
	(in particular when~${\omega = b}$).

	We now introduce the class of codes studied in this paper and 
	which contains those of type~\eqref{prop:eq:cbc}.
	 
	\begin{definition}
		We call~$n$-\DEF{complete bayonet code} ($n$-\textbf{cbc}) a
		set~${X \subseteq a^{\entiers} b a^{\entiers}}$ such that
		\begin{equation*}
			\CARD{X} = n 
			\hspace{5pt} \text{ and } \hspace{5pt}
			\ENS{a^n} \cup X \text{ is a code}.
		\end{equation*}
	\end{definition}
	
	Thus for any finite maximal code~$M$ containing~$a^n$ and for any word~$\omega$, 
	the set~$C_M(\omega)$ is an~$n$-cbc. However, 
	we do not know whether or not to any~$n$-cbc~$X$ corresponds a finite maximal code~$M$ 
	and a word~$\omega$ such that~$X = C_M(\omega)$. 
	Lam has nevertheless shown~\cite{Lam} that any cbc of the form~$a^P b a^Q$, 
	where~$(P,Q)$ is a Haj\'{o}s factorization, is included in a finite maximal code.
	We improve this result in Theorem~\ref{completion:th1}. 

	One of our main motivation is to study the \textit{strong triangle conjecture} 
	based on \textit{triangle property}.
	
	\begin{definition}
		We say that an~$n$-cbc~$X$ satisfies the \DEF{triangle property} if 
		\begin{equation*} 
		\CARD{ \ENS{x \text{ such that } x \in X \text{ and } |x| \leq k} } 
		\leq k,
		\end{equation*}
		for all~$k \in \entiers$. The \DEF{strong triangle conjecture} states that every 
		cbc satisfy the triangle property. 
	\end{definition}
	
	The strong triangle conjecture implies the Zhang and Shum conjecture and therefore
	the triangle conjecture.

	\section{Stability}
	
	In this section, we introduce and study some operations on cbc and their framework 
	related to finite maximal codes.
	Firstly, we introduce a composition operation for cbc.	
	
	\begin{definition} Let~$X$ and~$Y$ be~$n$-cbc, we set
	\begin{equation*}
        X \circ_r Y := 
        \left\{ a^{i} b a^{\ell} \text{ such that } a^{i} b a^{j} \in X, 
        \, a^{k} b a^{\ell} \in Y, 
        \text{and } \MOD{j + k} = r \right\},
   \end{equation*}	
   for any~$r \in \entiers$.
   \end{definition}

	\begin{example}		
	For any~$n$-cbc~$X$ and~$k \in \entiers$,  
    \begin{equation*}
    	X \circ_k \ENS{a^{\MOD{k-i}} b a^i, i \in \entiers} = X.
    \end{equation*} 
	\end{example}	

	The operations~$\PAR{ \circ_r }_{r \geq 0}$ are associative. 
	Indeed, for any~$n$-cbc~$X, Y, Z$ and~$r_1, r_2 \in \entiers$, 
	\begin{equation*}
		\PAR{ X \circ_{r_1} Y } \circ_{r_2} Z 
		\text{ and }
		X \circ_{r_1} \PAR{ Y \circ_{r_2} Z }
	\end{equation*}	 
	are equal to
	\begin{equation*}
        \ENS{ a^{i} b a^{j} \text{ such that } a^{i} b a^{j_1} \in X, \, 
        a^{i_2} b a^{j_2} \in Y, \, 
        a^{i_3} b a^{j} \in Z, \,  
        \MOD{j_1 + i_2} = r_1,
        \text{and } \MOD{j_2 + i_3} = r_2 }.
   \end{equation*}	
   However, the product of two~$n$-cbc does not necessarily results in an~$n$-cbc.
	For example, 
	\begin{equation*}
		\ENS{b, ba} \circ_0 \ENS{b, ab} = \ENS{b}.
	\end{equation*}		

	We introduce the notion of \textit{compatibility} as a framework 
	that enables composition of cbc.
	
	\begin{definition}
	A set~$\EnsCal{E}$ of~$n$-cbc is said to be \DEF{compatible} if for 
	all~$X_1, \dots, X_{k+1} \in \EnsCal{E}$ and~$r_1, \dots, r_k \in \entiers$,
    \begin{equation*}
        X_1 \circ_{r_1} X_2 \circ_{r_2} X_3 \dots \circ_{r_k} X_{k+1}
    \end{equation*}
    is an~$n$-cbc.	
	\end{definition}

	The compatibility of a set can be tested thanks to a graph algorithm.
	Given an integer~$n \geq 1$ and a set~$\EnsCal{E}$ of bayonet codes, 
	we denote by~$\mathcal{G}_n\left(\EnsCal{E}\right)$ 
	the directed graph made of the set of vertices~$\entiers$ and arrows from~$k_1$ to~$k_2$
	if and only if there exists~$X \in \EnsCal{E}$ 
	and two different words~$a^{i_1}ba^{j_1}, a^{i_2}ba^{j_2} \in X$ such that 
	\begin{equation*}
		k_1 = \MOD{i_1 - i_2} \text{ and } k_2 = \MOD{j_2 -j_1}.
	\end{equation*}

	\begin{proposition} \label{prop:algo_compatibilite}
	A set~$\EnsCal{E}$ of~$n$-cbc is compatible
	if and only if the graph~$\mathcal{G}_n\PAR{\EnsCal{E}}$
    does \textbf{not} contain a non-empty path from~$0$ to~$0$.
	\end{proposition}	
	
	\begin{proof}
		A set~$\EnsCal{E}$ of~$n$-cbc is not compatible if and only if there 
		exists~$t > 1$,~$X_1, \dots, X_{t} \in \EnsCal{E}$, and
		\begin{equation} \label{preuve:eq:chemin_graphe_1}
			a^{i_1}ba^{j_1}, a^{k_1}ba^{\ell_1} \in X_1, \dots, 
			a^{i_t}ba^{j_t}, a^{k_t}ba^{\ell_t} \in X_{t}
		\end{equation}
		such that~$j_1 \neq \ell_1$ and
		\begin{equation} \label{preuve:eq:chemin_graphe_2}
			a^{i_1}ba^{j_1} a^{i_2}ba^{j_2} \dots a^{i_t}ba^{j_t} 
			\Mod 
			a^{k_1}ba^{\ell_1} a^{k_2}ba^{\ell_2} \dots a^{k_t}ba^{\ell_t}.
		\end{equation}
		Thus if~$\EnsCal{E}$ is not compatible then~$\mathcal{G}_n\PAR{\EnsCal{E}}$ 
		contains the path
		\begin{equation*}
		\begin{tikzpicture}[baseline=-0.5ex]
			\node[draw,rounded corners] (A) at (0, 0) {$0 = \MOD{i_1 - k_1}$};
			\node[draw,rounded corners] (B) at (3.4, 0) {$\MOD{\ell_1-j_1} = \MOD{i_2-k_2}$};
			\node (C) at (5.8, 0) {$\dots$};
			\node[draw,rounded corners] (D) at (8.6, 0) {$\MOD{\ell_{t-1}-j_{t-1}} = \MOD{i_t-k_t}$};
			\node[draw,rounded corners] (E) at (12.2, 0) {$\MOD{\ell_t-j_t} = 0$};

			\draw[->] (A) edge (B) ;
			\draw (B) edge (C) ;
			\draw[->] (C) edge (D) (D) edge (E);
		\end{tikzpicture}\,.
		\end{equation*}
		
		Conversely, for any path from~$0$ to~$0$ of length~$t > 1$ in the 
		graph~$\mathcal{G}_n\left(\EnsCal{E}\right)$,
		there exists~\eqref{preuve:eq:chemin_graphe_1} such that the
		 relation~\eqref{preuve:eq:chemin_graphe_2} is true.
	\end{proof}
	
	We can see~$\mathcal{G}_n\left(\EnsCal{E}\right)$ as the superposition 
	of graphs~$\mathcal{G}_n\left(\ENS{X}\right)$, where~$X \in \EnsCal{E}$.
	So in the particular case where~$\EnsCal{E}$ is a singleton, 
	the graph~$\mathcal{G}_n\left(\EnsCal{E}\right)$
	is equivalent to the graph defined in Proposition~1 
	of~\cite{perrin1981conjecture}\footnote{In Proposition~1 of~\cite{perrin1981conjecture},
	Perrin and Schützenberger defined a graph on sets of integers. 
	They did not explain the link with theory of codes.  
	However, one can understand it as an algorithm to test whether or not a set of 
	the form~$\ENS{a^n} \cup X$ (where~$X \subseteq a^*ba^*$) is a code.
	Their graph is not equal to~$\mathcal{G}_n\PAR{\ENS{X}}$ in general but it also contains a 
	non-empty path from~$0$ to~$0$ if and only if the considered set is not a code.} 
	and equal to the graph~$\mathcal{G}_{mod}$ of~\cite{cordero2019note}.

	We introduce the \textit{stable} notion as compatible sets
	closed under composition.
	
	\begin{definition}
	A set~$\EnsCal{S}$ of~$n$-cbc is \DEF{stable} if for all~$X, Y \in S$ 
	and~$r \in \entiers$,
    \begin{equation*}
        X \circ_{r} Y \in \EnsCal{S}.
    \end{equation*}
	\end{definition}     
	
	A stable set is therefore compatible.
	 Given a set~$\EnsCal{E}$ of compatible cbc, we denote by~$\EnsCal{E}^\circ$ 
	the smallest stable set (for inclusion) containing~$\EnsCal{E}$.
	We say that~$\EnsCal{E}^\circ$ is the stable of~$\EnsCal{E}$. 
	For example, we obtain	
    \begin{equation} \label{ex:code_max_stable}
    	\ENS{ C_E(b), C_E(bb)}^\circ = \ENS{ C_E(b), C_E(bb), 
    	\ENS{ ba, aba, a^2b, a^3ba },
		\ENS{ ba, aba, a^2ba, a^3ba }},
    \end{equation} 
    where~$E$ is the code~\eqref{ex:code_max}.
    
    We can associate stable sets to any finite maximal code.
	
	\begin{proposition}	 \label{Prop_code}
		Let~$M$ be a finite maximal code, the sets 
		\begin{equation*}
   			\EnsCal{C}_M := \ENS{ C_M(\omega), \, \omega \in \Alpha^*}
   			\text{ and }
   			\EnsCal{C}_M' := \ENS{ C_M(\omega), \, 
   			\omega \in \B\PAR{a^*\B}^*},
   		\end{equation*}
   		where~$\B$ is the alphabet~$\Alpha\setminus\ENS{a}$, are stable.
	\end{proposition}    
    
    \begin{proof}
    Let~$M$ be a finite maximal code and~$n$ the order of~$a$. 
    According to Proposition~\ref{prop:cbc}, 
	the set~$\EnsCal{C}_M$ (respectively~$\EnsCal{C}_M'$) is a set of~$n$-cbc.
    Moreover, for all~$C_1, C_2 \in \EnsCal{C}_M$ (resp.~$\EnsCal{C}_M'$), there 
    exists~$\omega_1, \omega_2 \in \Alpha^*$ (resp.~$\B\PAR{a^*\B}^*$) such 
    that~$C_1 = C(\omega_1)$ and~$C_2 = C(\omega_2)$. 
    For any~$r \in \entiers$, we have  
    \begin{equation*}
    	C_1 \circ_r C_2 = C_{M}(\omega_1 a^{r+tn} \omega_2) \in \EnsCal{C}_M 
    	 \hspace{5 pt} \PAR{\text{resp. } \EnsCal{C}_M'}, 
    \end{equation*}
    for~$t \geq \frac{2}{n}\underset{m \in M}{\max}\ENS{|m|} - \frac{r}{n}$.
    \end{proof}
    
    For example, we obtain that the set~$\EnsCal{C}_E'$ is equal to the 
    set~\eqref{ex:code_max_stable}, when~$E$ is the set~\eqref{ex:code_max}.
    
    Proposition~\ref{Prop_code} thus gives a criterion that a code must 
    satisfies in order to be included in a finite maximal code.
    
    \begin{example} Suppose that the code
    \begin{equation} \label{Stable:ex:non_inclus}
    	\ENS{a^5,ab,aba^2, a^2b^2,ab^2}
    \end{equation}
    is included in a finite maximal code~$M$. 
    One can notice that the code~\eqref{Stable:ex:non_inclus} satisfies the necessary 
    conditions implied by 
    Theorem~\ref{Preliminaire:th:schutz} and Proposition~\ref{prop:cbc}.
    For example, the codes
	\begin{equation} \label{Stable:ex:critere}
		 \ENS{ab,aba^2} \subseteq C_M(b)
		 \text{ and } 
		 \ENS{a^2b,ab} \subseteq C_M(bb) 	
   	\end{equation}    
   	are respectively included in the~$5$-cbc 
   	\begin{equation*}
		\ENS{ab,aba^2,ba^4, ba^3, ba} 
		\text{ and } 
		\ENS{a^2b,ab, a^4b, a^3b, b}.   	
   	\end{equation*}
   	
   	However, thanks to  an exhaustive computer exploration of the (finite) set of~$5$-cbc,
   	we found, using the algorithm from Proposition~\ref{prop:algo_compatibilite}, 
   	that none of the cbc containing~\eqref{Stable:ex:critere}   	
   	are compatible.
   	Thus, according to Proposition~\ref{Prop_code}, the 
   	code~\eqref{Stable:ex:non_inclus} is not 
   	included in a finite maximal code. 
    \end{example}

	\section{Border}

	In this section, we first show that each stable set can be associated 
	to a notion that we call \textit{border}. 
	It is a generalization of a notion 
	originally introduced on codes satisfying the factorization conjecture 
	in section~5.3 of~\cite{Felice}.
	It exhibits a link between factorizations and stable sets. 
	Secondly, we show various operations to build \textit{borders} from others. 
	
	\subsection{Definition and existence}

	\begin{definition}
	Given~$P, Q \subseteq \mathbb{Z}$, we say that the ordered 
	pair~$\PAR{P, Q}$ is a \DEF{border} 
	of an~$n$-cbc~$X$ if
	\begin{equation*}
   		a^P \, \SOM{X} \, a^Q \Mod a^{\entiers}ba^{\entiers}
   		\hspace{5pt}\PAR{\Mod \sum\limits_{i,j \in \entiers} a^iba^j}.
   	\end{equation*}
   	More generally, we say that it is a border of a set of cbc if it borders
   	each of its elements.
	\end{definition}

	\begin{example}
		The ordered pair 
		\begin{equation}\label{exemple:bord}
			\PAR{\ENS{ 2, 4 } , \ENS{ 0, 2, 4, 6 }}
		\end{equation}
		and the factorization~\eqref{exemple:factorisation}
		are borders of the~$8$-cbc
		\begin{equation} \label{exemple:cbc}
			\ENS{b, ba, aba^2, a^3ba^3, a^4b, 
			a^4ba, a^5ba^2, a^7ba^7 }.
		\end{equation}
		The factorization
		\begin{equation} \label{exemple:bord_Krasner}
			\PAR{\ENS{ 0 } , \ENS{ 0, 1, 2, 3 }}
		\end{equation}
		is a border of the stable set~\eqref{ex:code_max_stable}.
	\end{example}	
	
	Given an~$n$-cbc~$X$, note that~$(P,Q)$ borders~$X$ if and only if~$(Q,P)$ borders
	its dual
	\begin{equation*}
		\DUAL{X} := \ENS{a^jba^i, a^iba^j \in X}
	\end{equation*}	
	and that a factorization~$(P,Q)$ borders~$\ENS{X}^\circ$ 
	if and only if
	\begin{equation*}
		a^P \frac{1}{1-\SOM{X}} a^Q \Mod \SOM{\Alpha^*_{/a^n = \varepsilon}},
	\end{equation*}
	where~$\Alpha^*_{/a^n = \varepsilon}$ is the 
	quotient of~$\Alpha^*$ by the 
	relation~$a^n = \varepsilon$.

	We write~$x \dans X$ if and only if there 
	exists~$y \in X$ such that~$x \Mod y$.
	For any~$Y \subseteq a^*ba^*,n \geq 1$, and~$k \in \entiers$, we set
	\begin{equation*}
		\L{k}{n}{Y} := \ENS{\MOD{i}, \, a^iba^k \dans Y}, \,
   		L(Y) := \ENS{\MOD{i}, \, a^iba^j \in Y},
	\end{equation*}	
	and
	\begin{equation*}
		\R{k}{n}{Y} := \ENS{\MOD{j}, \, a^kba^j \dans Y}, \,	
		R(Y) := \ENS{\R{k}{n}{Y}, \, k \in L(Y)}.
    \end{equation*}
    For example, $L(X) = \ENS{0,1,3,4,5,7}$ and~$R(X) = \ENS{\ENS{0,1},\ENS{2},\ENS{3},\ENS{7}}$ 
    when~$X$ is~\eqref{exemple:cbc}.
	We say that an~$n$-cbc~$X$ \DEF{borders} a set~$\EnsCal{E}$ if and only if for 
	any~$R \in R(X)$, the ordered pair~$\PAR{R, L(X)}$ is a factorization of 
	size~$n$ that borders~$\EnsCal{E}$.

	\begin{proposition} \label{prop:stable}
	Any stable set is bordered by one of its elements.
	\end{proposition}
	
	\begin{proof}
		Let~$\EnsCal{S}$ be a stable set of~$n$-cbc. 
		If~$Y \in \EnsCal{S}$ does not borders~$\EnsCal{S}$ then there  
		exists~$i,j \in \entiers$, $k \in L(Y)$, and~$X \in \EnsCal{S}$ such that 
		\begin{equation*}
			a^i b a^j \not\dans a^{\R{k}{n}{Y}} \, X \, a^{L(Y)} 
			\text{ or } a^i \not\dans a^{\R{k}{n}{Y}} a^{L(Y)}.
		\end{equation*}
		Thus~$Y' := Y \circ_i X \circ_j Y$ or~$Y' := Y \circ_i Y$ is a cbc belonging 
		to~$\EnsCal{S}$ such that~$\CARD{L(Y')} < \CARD{L(Y)}$, 
		because~$k \in L(Y)$ and~$k \notin L(Y')$. 
		
		The cardinality of a set being a positive integer, 
		we obtain the expected result by iterating this process at most~$\CARD{L(X)}$ times 
		starting from any element~$X$ of~$\EnsCal{S}$.
	\end{proof}
	
	Thus any stable set admits a bordering factorization and conversely for any factorization, 
	there exists a cbc that it borders. Indeed, a factorization~$(P, Q)$
	 borders the cbc~$a^Qba^P$.

	It is possible to build, by compositions of elements of a compatible set, a cbc 
	in a more restrictive form which borders the set.
	
	\begin{theorem} \label{th:bord:FN}
	Let~$\EnsCal{E}$ be a compatible set of~$n$-cbc. For all~$Z \in \EnsCal{E}$, there 
	exists~$r_1, \dots, r_{k} \in \entiers$
	and~$X_1, \dots,$ $X_{k} \in \EnsCal{E}$  such that the cbc
	\begin{equation*} 
		X := Z \circ_{r_1} X_1 \circ_{r_2} \dots \circ_{r_{k}} X_{k}
	\end{equation*}
	borders~$\EnsCal{E}$ and such that for all~$R, R' \in R(X)$,
	\begin{equation*}
		R \cap R' \neq \emptyset \implies R = R'.
	\end{equation*}
	\end{theorem}

	\begin{proof}
		Let~$\EnsCal{E}$ be a compatible set of cbc and~$Y \in \EnsCal{E}^\circ$
		a cbc bordering~$\EnsCal{E}$.
		If there exists~$r \in R \cap R'$, such that~$R \neq R'$ and~$R, R' \in R(Y)$ then 
		for any~$\ell \in L(Y)$, the cbc
		\begin{equation*}
			Y' := Y \circ_{\MOD{r+\ell}} Y
		\end{equation*}	
		satisfies the facts that~$R(Y')$ is strictly included in~$R(Y)$ and that~$L(Y') = L(Y)$.
		Thus~$Y' \in \EnsCal{E}^\circ$ and~$Y'$ also borders~$\EnsCal{E}$.
		
		The set~$R(X)$ of a cbc~$X$ cannot be empty, so we obtain the expected result 
		by iterating this process at most~$\CARD{R(Y)}$ times
		starting from 
		a cbc~$Y$ given by Proposition~\ref{prop:stable}. 
		Moreover, according to the proof of Proposition~\ref{prop:stable}, such~$Y$ can be chosen  
		by starting from any~$Z \in \EnsCal{E}$.
		This concludes the proof.
	\end{proof}		
	
	One of the consequences of Theorem~\ref{th:bord:FN} is that for any finite maximal code~$M$
	containing~$a^n$ and any word~$\omega_1 \in \Alpha^*$, 
	there exists a word~$\omega \in \Alpha^*$ such that 
	\begin{equation*}
        \SOM{C_M\PAR{\omega_1 \omega}} \Mod \sum\limits_{k \in \ENTIERS{t}} a^{L_k}ba^{R_k},
   	\end{equation*}
   	where~$R_i \cap R_j = \emptyset$ when~$i \neq j$, and such that
   	\begin{equation*}
   	 \PAR{ L_1 \sqcup \dots \sqcup L_t, R_k }_{k \in \ENTIERS{t}}
   	\end{equation*}
    are factorizations bordering~$\EnsCal{C}_M$.
	
	\subsection{New borders from old ones}
	
	Being a border is closed under translations and multiplications.
	
	\begin{proposition} \label{prop:bord_trans}
		If~$(P, Q)$ is a border of a cbc~$X$ 
		then for all~$i, j \in \mathbb{Z}$,~$d_1$ prime to~$\CARD{P}$, 
		and~$d_2$ prime to~$\CARD{Q}$, the ordered 
		pairs
		\begin{equation*}
			(P+i,Q+j) \text{ and } (d_1 P, d_2 Q) 
		\end{equation*}
		are borders of~$X$.
	\end{proposition}
	
	\begin{proof}
		First, we recall some properties about factorizations.
		It is well known that for any~$j \in \mathbb{Z}$,
		an ordered pair~$(P,Q)$ is a factorization of size~$n$ if and only if~$(P,Q + j)$ 
		is a factorization of size~$n$, since for all~$p_1, p_2 \in P$ 
		and~$q_1, q_2 \in Q$, we have
		\begin{equation*}
			\MOD{p_1 + \PAR{q_1 + j}} = 
			\MOD{p_2 + \PAR{q_2 + j}} \iff
			\MOD{p_1 + q_1} = 
			\MOD{p_2 + q_2}.
		\end{equation*}
		Moreover, according to Proposition~3 of~\cite{Sands}, 
		if~$(P,Q)$ is a factorization and~$d$ a number prime to~$\CARD{Q}$ then 
		the ordered pair~$(P,d Q)$ is a factorization.
		
		An ordered pair~$(P,Q)$ borders an~$n$-cbc~$X$ if and only if for 
		any~$k \in \entiers$, the ordered pair
		\begin{equation*}
			\PAR{ \R{k}{n}{a^P X}, Q }
		\end{equation*}
		is a factorization of size~$n$. 
		Assume that~$(P,Q)$ borders~$X$ then, according to the previous recalls, 
		for any~$k \in \entiers$,~$j \in \mathbb{Z}$, and~$d$ prime to~$\CARD{Q}$, 
		the ordered pair 
		\begin{equation*}
			\PAR{ \R{k}{n}{a^P X}, j + d Q }
		\end{equation*} 
		is a factorization. Thus~$(P, j + d Q)$ borders~$X$.
		We obtain the expected result thanks to a symmetrical argument.
	\end{proof}

	In some cases, Proposition~\ref{prop:bord_trans} enables to explicitly compute a border.
	
	\begin{proposition} \label{prop:bord:ss_groupe}
		If an~$n$-cbc~$X$ is bordered by~$(P,Q)$ 
		such that~$p := |P|$ is prime to~$q := |Q|$ then the factorization 
		\begin{equation} \label{eq:bord:premier}
			\PAR{ q \ENTIERS{p}, \, p\ENTIERS{q} }
		\end{equation} 
		borders~$X$.
	\end{proposition}
	
	\begin{proof}
		Assuming the hypotheses of the Proposition,
		the number~$q$ is prime to~$p$ so according to Proposition~\ref{prop:bord_trans}, 
		the ordered pair~$\PAR{q P,Q}$ borders~$X$. 
		The set~$q P$ contains~$p$ distinct elements, all of which are multiples of~$q$.
		Thus
		\begin{equation*}
		qP = \ENS{0,q, \dots, q(p-1)} = q\ENTIERS{p}.
		\end{equation*}			
		Symmetrically, we obtain that the ordered pair~\eqref{eq:bord:premier} borders~$X$. 
		Moreover, according to Bézout's identity~\cite{bezout1779theorie}, 
		there exists~$u,v \in \mathbb{Z}$ such 
		that~${up+vq = 1}$. Thus for all~$k \in \entiers$, we have
		\begin{equation*}
			k = \MOD{p(ku) + q(kv)} \in p\ENTIERS{q} + q \ENTIERS{p}.
		\end{equation*}
		Thus the ordered pair~\eqref{eq:bord:premier} is a factorization.
	\end{proof}
	
	The composition of cbc from a stable set brings out bordering factorizations.
	
	\begin{theorem} \label{th:bordante}
    	Let~$\EnsCal{S}$ be a stable set of~$n$-cbc and~$(P,Q)$ one of its borders. 
    	For all~${k_1,k_2 \in \entiers}$ and~$X,Y \in \EnsCal{S}$,
		the ordered pairs
		\begin{equation*}
			\PAR{P, \, \L{k_2}{n}{Y a^Q}},\, 
			\PAR{\R{k_1}{n}{a^P X},\,\L{k_2}{n}{Y a^Q}},
			\text{ and }
			\PAR{ \R{k_1}{n}{a^P X},\, Q }
		\end{equation*}
		 are factorizations bordering~$\EnsCal{S}$.
    \end{theorem}  
    
    \begin{proof}
		Assuming the hypotheses of the Theorem,
    	if the ordered pair 
    	\begin{equation} \label{preuve:eq:couple}
    		\left(\R{k_1}{n}{a^P X},\,\L{k_2}{n}{Y a^Q}\right), 
			\text{ respectively } 
			\left(P, \, \L{k_2}{n}{Y a^Q}\right),
		\end{equation}   
		is not a factorization 
    	then there exists~$i \in \entiers$ which is not generated by the 
    	sum of its two components (modulo~$n$) and thus
    	\begin{equation*}
			a^{k_1} b a^{k_2} \not\dans a^P X \circ_i Y a^Q, 
			\text{resp. } a^{i} b a^{k_2} \not\dans a^P  Y a^Q.
		\end{equation*}
		Thus~$(P,Q)$ is not a border of~$\EnsCal{S}$.
    	Which contradicts the assumptions.
    	
    	Likewise, if the ordered pair~\eqref{preuve:eq:couple} does not borders~$\EnsCal{S}$
    	then there exists~$Z \in \EnsCal{S}$ and~$i,j \in \entiers$ such that
    	\begin{equation*}
    		 a^iba^j \not\dans a^{\R{k_1}{n}{a^P X}}
			Z a^{\L{k_2}{n}{Y a^Q}}, \,
			\text{resp. } 
			a^iba^j \not\dans a^P Z
			a^{\L{k_2}{n}{Y a^Q}}.
		\end{equation*} 
		Which implies that
		\begin{equation*}
    		a^{k_1}ba^{k_2} \not\dans a^P X \circ_i Z \circ_j Y a^Q, \,
			\text{resp. } 
			a^{i}ba^{k_2} \not\dans a^P Z \circ_j Y a^Q.
		\end{equation*}
		Thus~$(P,Q)$ does not borders~$\EnsCal{S}$.
    	Which contradicts the assumptions.
    	
    	We obtain the last case thanks to a symmetrical argument.
	\end{proof}

    \section{Haj\'{o}s cbc} \label{HajosCbc}
    
	In this section, similarly to factorization theory, we introduce a \textit{periodic} 
	and a \textit{Haj\'{o}s} notion 
	for cbc and compatible sets. Then we show that this \textit{Haj\'{o}s} notion 
	is equivalent as being bordered by  a
	\textit{Krasner factorization}.    
    
	\subsection{Periodic cbc}
    
    We introduce an operation to build bigger cbc from a smaller one.
	Given a set
   		\begin{equation*}
   			X := \ENS{a^{i_1}ba^{j_1}, \dots, a^{i_n}ba^{j_n}}
   			\subseteq a^{\entiers}ba^{\entiers}
   		\end{equation*}
   		and~$t \geq 1$, we define the operation 
   		\begin{equation*}
   			H_{t}\PAR{X} := \ENS{
   			\bigsqcup \limits_{\ell = 1}^{n}
   			\ENS{a^{i_\ell + k_{\ell,1} n}ba^{j_\ell}, \dots, 
   			a^{i_\ell + k_{\ell,t} n}ba^{j_\ell + (t-1)n}},
   			k_{1,1}, \dots, k_{n,t} \in \ENTIERS{t}}.
   		\end{equation*} 
   	For example, the dual of~\eqref{exemple:cbc} 
    is equal to
		\begin{equation} \label{exemple:hajos_cbc_2}
			\ENS{ 
    		\begin{tabular}{cc}
   				$a^{0+0 \times 4}ba^{0}$ & $a^{0+0 \times 4}ba^{0+1 \times 4}$ \\
   				$a^{1+0 \times 4}ba^{0}$ & $a^{1+0 \times 4}ba^{0+1 \times 4}$ \\
   				$a^{2+0 \times 4}ba^{1}$ & $a^{2+0 \times 4}ba^{1+1 \times 4}$ \\
   				$a^{3+0 \times 4}ba^{3}$ & $a^{3+1 \times 4}ba^{3+1 \times 4}$ \\
			\end{tabular}} \in H_2\PAR{\ENS{b, ab, a^2ba, a^3ba^3}}.
    	\end{equation}	
    	
    We extend this operation to compatible sets. 
    We write~$\EnsCal{E'} \in \EnsCal{H}_{t}\PAR{\EnsCal{E}}$,
    where~$t \geq 1$, if and only if
    \begin{equation*}
    	Y \in \EnsCal{E'} \implies  \exists X \in \EnsCal{E} 
    	\text{ such that } Y \in H_t\PAR{X}
    \end{equation*}
    and   
	\begin{equation*}
    	X \in \EnsCal{E} \implies  \exists Y \in \EnsCal{E'} 
    	\text{ such that } Y \in H_t\PAR{X}.
    \end{equation*}
    For example, if~$\EnsCal{S}$ is the stable set~\eqref{ex:code_max_stable}
    associated to the code~\eqref{ex:code_max} 
    then~$\DUAL{\EnsCal{S}} \in \EnsCal{H}_4\PAR{\ENS{\ENS{b}}}$,
    where~$\DUAL{\EnsCal{S}}$ is the set~$\ENS{\DUAL{X}, \, X \in \EnsCal{S}}$.
    
	This operation preserve the fact of being a compatible set.    
    
	\begin{proposition} \label{proposition_cbc_period}
    	Given~$t \geq 1$ and~$\EnsCal{E'} \in \EnsCal{H}_{t}\PAR{\EnsCal{E}}$,
    	the set~$\EnsCal{E}$ is a compatible set of~$n$-cbc such that~$\EnsCal{E}^\circ$ 
    	is bordered by~$\PAR{P,Q}$ 
    	if and only 
    	if~$\EnsCal{E'}$ is a compatible set of~$nt$-cbc such that~$\EnsCal{E'}^\circ$ 
    	is bordered by~$\PAR{P+n\ENTIERS{t},Q}$.
    \end{proposition}    
     
	\begin{proof}
		Let~$\EnsCal{E}$ be a compatible set of~$n$-cbc whose stable is bordered by~$\PAR{P,Q}$ 
		and~$\EnsCal{E'} \in \EnsCal{H}_{t}\PAR{\EnsCal{E}}$.

		For all~$Y_1, \dots, Y_k \in \EnsCal{E'}$, there 
		exists~$X_1, \dots, X_k \in \EnsCal{E}$ such that~$Y_i \in H_t\PAR{X_i}$, 
		when~$i \in [1,k]$.
		Since 
   		\begin{equation*}
   			a^{n\ENTIERS{t}} \SOM{Y_i} 
   			\Modn{nt} a^{n\ENTIERS{t}} \SOM{X_i} a^{n\ENTIERS{t}} 
   		\end{equation*}
   		 for all~$i \in [1,k]$, we have that
		\begin{equation} \label{conversion_nt_n_1}
			a^{P+n\ENTIERS{t}}\SOM{Y_1} \cdots \SOM{Y_k} a^{Q} \Modn{nt} 
			a^{P+n\ENTIERS{t}} \PAR{\SOM{X_1} a^{n\ENTIERS{t}}} 
   			\cdots \PAR{\SOM{X_k} a^{n\ENTIERS{t}}} a^{Q}.
   		\end{equation}	
   		
   		Moreover, according to the hypothesis,
		\begin{equation}  \label{conversion_nt_n_4}
			a^{P} \SOM{X_1} \cdots \SOM{X_k} a^{Q}
   			\Modn{n} 
    		\PAR{a^{\ENTIERS{n}}b}^k a^{\ENTIERS{n}}
   		\end{equation}	
   		and since for all~$i,j$,
    	\begin{equation}  \label{conversion_nt_n_2} 		
    		a^i \Modn{n} a^j \iff a^{i+n\ENTIERS{t}} \Modn{nt} a^{j+n\ENTIERS{t}},
		\end{equation} 
    	we have that
   		\begin{equation} \label{conversion_nt_n_5} 
    		a^{P+n\ENTIERS{t}} \PAR{\SOM{X_1} a^{n\ENTIERS{t}}} 
   			\cdots \PAR{\SOM{X_k} a^{n\ENTIERS{t}}} a^{Q}
    		\Modn{nt}
    		 \PAR{a^{\ENTIERS{n} + n\ENTIERS{t}}b}^k a^{\ENTIERS{n} + n\ENTIERS{t}}
    		\Modn{nt}
    		\PAR{a^{\ENTIERS{nt}}b}^k a^{\ENTIERS{nt}}
    	\end{equation}
    	and thus 
    	\begin{equation}  \label{conversion_nt_n_3}
    		a^{P+n\ENTIERS{t}}\SOM{Y_1} \cdots \SOM{Y_k} a^{Q} 
    		\Modn{nt}
    		\PAR{a^{\ENTIERS{nt}}b}^k a^{\ENTIERS{nt}}.
    	\end{equation}
    	This shows that~$\EnsCal{E'}$ is a compatible set of~$nt$-cbc whose stable is bordered 
    	by~$\PAR{P+n\ENTIERS{t},Q}$. 
   		
   		Conversely, let~$\EnsCal{E'}$ be a compatible set of~$nt$-cbc
   		whose stable is bordered by~$(P+n\ENTIERS{t},Q)$ and such 
   		that~$\EnsCal{E'} \in \EnsCal{H}_{t}\PAR{\EnsCal{E}}$.
   		For all~$X_1, \dots, X_k \in \EnsCal{E}$, there 
		exists~$Y_1, \dots, Y_k \in \EnsCal{E'}$ such that~$Y_i \in H_t\PAR{X_i}$, 
		when~$i \in [1,k]$.
		We have by hypothesis~\eqref{conversion_nt_n_3} and~\eqref{conversion_nt_n_1} 
		thus~\eqref{conversion_nt_n_5}.
		We apply~\eqref{conversion_nt_n_2} to~\eqref{conversion_nt_n_5} in order 
		to get~\eqref{conversion_nt_n_4}.
		
    	This concludes the proof.
    \end{proof}          
	
    We introduce a periodic notion for cbc and compatible sets.    
    
    \begin{definition} 
   		We say that~$Y$ is~$n$-\DEF{right-periodic} if there exists 
   		an~$n$-cbc~$X$ and~$t > 1$, such that~$Y \in H_{t}\PAR{X}$ and 
   		we say that~$Y$ is~$n$-\DEF{periodic} if~$Y$ or~$\DUAL{Y}$ 
   		is~$n$-\DEF{right-periodic}. 
   		
   		More generally, we say that a compatible set 
   		is~$n$-\DEF{right-periodic} if all its elements are~$n$-right-periodic. 
    \end{definition}
   
    For example, the cbc~\eqref{exemple:cbc} is~$4$-periodic since 
    \begin{equation} \label{exemple:hajos_cbc_1}
    	\ENS{b, ab, a^2ba, a^3ba^3}
    \end{equation}
	is an~$4$-cbc and~\eqref{exemple:hajos_cbc_2}.
    	
	\begin{remark} \label{remarque:hajos_modulo}
		Note that if~$Y$ is an~$n$-\DEF{right-periodic}~$nt$-cbc 
		then~$Y \in H_{t}\PAR{\MODn{Y}{n}}$, where 
		\begin{equation*}
		\MODn{Y}{n} = \ENS{a^{\MODn{i}{n}} b a^{\MODn{j}{n}}, \, a^i b a^j \in Y }.
		\end{equation*}
	\end{remark}
    	
	The next proposition links periodicity of factorizations to periodicity of compatible sets.
	
	\begin{proposition} \label{prop:critere_periodique}
		Let~$\EnsCal{E}$ be a compatible set of~$nt$-cbc such that~$\EnsCal{E}^\circ$
		is bordered by~$\PAR{P+n\ENTIERS{t},Q}$.
		If for all~${k \in \ENTIERS{nt}}$ and~$Y \in \EnsCal{E}$,
		the sets
		\begin{equation} \label{prop:critere_periodique_eq}
			\R{k}{nt}{a^{P} Y}
		\end{equation}			
		are~$n$-periodic in~$\Zn{nt}$ then~$\EnsCal{E}$ is~$n$-right-periodic.
	\end{proposition}
	
	\begin{proof}
		For all~$Y \in \EnsCal{E}$, if the sets~\eqref{prop:critere_periodique_eq} 
		are~$n$-periodic then
   		\begin{equation} \label{local:1}
			a^{P+n\ENTIERS{t}}Y \Modn{nt} 
   			a^{P+n\ENTIERS{t}} \MODn{Y}{n} a^{n\ENTIERS{t}}
   		\end{equation}
   		is unambiguous and thus~$\CARD{\MODn{Y}{n}} = n$. Moreover,  
		if~$a^{i + k_1 n}ba^{j}, a^{i + k_2 n}ba^{j} \in Y$, 
		where~$i < n$ and~$k_1,k_2 < t$, then for any~$p \in P$,
		\begin{equation*}
			a^{p+k_2 n} a^{i + k_1 n}ba^{j} 
			\Modn{nt} 
			a^{p+k_1 n} a^{i + k_2 n}ba^{j} 
		\end{equation*}
		and since~\eqref{local:1} is ambiguous it implies that~$k_1 = k_2$.
		Thus if~$a^iba^j \in \MODn{Y}{n}$ then there exists~$k_1, \dots k_t \in \ENTIERS{t}$
		such that
		\begin{equation*}
			\ENS{a^{i + k_1 n}ba^{j}, \dots, a^{i + k_t n}ba^{j+(t-1)n}} \subseteq Y.
		\end{equation*}   
		Therefore~$Y \in H_t\PAR{\MODn{Y}{n}}$. 
		
		Moreover, according to Proposition~\ref{proposition_cbc_period},~$\MODn{Y}{n}$ 
		is an~$n$-cbc. So~$Y$ (and thus~$\EnsCal{E}$) is~$n$-right-periodic.
    	This concludes the proof.
	\end{proof}

	We define a \textit{Haj\'{o}s} notion for cbc as composition of periodic cbc. Formally, 	
    we denote by~$H_n$ the set of \DEF{Haj\'{o}s cbc} of size~$n$ that we recursively 
    define as follows:
    \begin{equation*}
    H_n := \left\{\begin{array}{cl}
    	\vspace{2pt} \ENS{\ENS{b}} & \text{if } n = 1,\\
      \bigcup \limits_{\substack{n = tm, \, t > 1,\vspace{2pt} \\ X \in H_m}} 
      \DUAL{H_t\PAR{X}} \cup H_t\PAR{X}
       & \text{otherwise}.
    \end{array}\right.
    \end{equation*}	
    Since~$\ENS{b}$ is an~$1$-cbc then, according to Proposition~\ref{proposition_cbc_period}, 
    a Haj\'{o}s cbc is a cbc.
    For example, the cbc~\eqref{exemple:cbc} is a Haj\'{o}s cbc since~\eqref{exemple:hajos_cbc_2}
	and the dual of~\eqref{exemple:hajos_cbc_1} is equal to
	\begin{equation*}
		\ENS{a^{0+0 \times 1}ba^{0},a^{0+0 \times 1}ba^{0+1 \times 1},
			 a^{0+1 \times 1}ba^{0+2 \times 1},a^{0+3 \times 1}ba^{0+3 \times 1}} 
			 \in H_4\PAR{\ENS{b}}.
	\end{equation*}

	We extend the Haj\'{o}s notion to compatible sets.
	A compatible set~$\EnsCal{E}$ is said to be of \DEF{Haj\'{o}s} if and only if
	there exists~$t_1, \dots, t_k > 1$ and some cbc~$Y_1, \dots, Y_{k-1}$
	 such that for all~$Y \in \EnsCal{E}$ or 
	for all~$Y \in \DUAL{\EnsCal{E}}$,
	\begin{equation*}
		Y \in H_{t_k}\PAR{Y_{k-1}},\,
		\DUAL{Y_{k-1}} \in H_{t_{k-1}}\PAR{Y_{k-2}},\,
		\dots,
		\DUAL{Y_{2}}    \in H_{t_2}\PAR{Y_1},\,
		\DUAL{Y_1} \in H_{t_1}\PAR{\ENS{b}}.
	\end{equation*}
    Note that, according to Remark~\ref{remarque:hajos_modulo},
    we necessarily have~$Y_i = \MODn{Y}{t_{1}\cdots t_{i}}$, for all~$1 \leq i < k$.
    In particular,~$\EnsCal{E}$ is of Haj\'{o}s if and only if~$\DUAL{\EnsCal{E}}$ 
    is of Haj\'{o}s.
	
	\subsection{Krasner border}
	
	We first do some recalls about \textit{Krasner factorizations}. 
	An ordered pair~$\PAR{P,Q}$ is a \DEF{Krasner factorization} 
	(of size~$n := \CARD{P} \times \CARD{Q}$)  
	if and only if for	all~$k \in \ENTIERS{n}$,
	there exists~$p \in P$ and~$q \in Q$ such that 
	\begin{equation*}
		k = p + q.
	\end{equation*}  
	For example, the ordered pair~\eqref{exemple:bord_Krasner} 
	is a Krasner factorization of size~$4$.
	
	Krasner factorizations are completely described in~\cite{krasner1937propriete}. 
	For all~$t_1, \dots, t_k > 1$, the ordered pairs~$\PAR{U,V}$ and~$\PAR{V,U}$, where
	\begin{equation} \label{Enum_Krasner}
		U := \sum \limits_{i \in [1,k], \, 2 \mid i} t_1 \dots t_{i-1} \ENTIERS{t_i}
		\text{ and }
		V := \sum \limits_{i \in [1,k], \, 2 \nmid i} t_1 \dots t_{i-1} \ENTIERS{t_i},
	\end{equation}
	are Krasner factorizations of size~$t_1 \cdots t_k$.
	Conversely, any Krasner factorization can be built that way.
	
	Krasner factorizations naturally appear in the factorization conjecture as shown in~\cite{Felice}
	and in Proposition~3.6 of~\cite{Felice2022}.
	
	\begin{proposition} \label{prop:conj_facto_bord_krasner}
	If a finite maximal code~$M$ satisfies the 
    factorization conjecture then
    the set~$\EnsCal{C}_M$ is bordered by a Krasner factorization.
	\end{proposition}
	
	Our statement is slightly different, we provide a straightforward proof. 	
	
	\begin{proof}
	If a finite maximal code~$M$ satisfies the factorization conjecture then 
    there exists~$P,S \subseteq \Alpha^*$ such that
    \begin{equation} \label{conjecture:krasner_factorisation}
    	\SOM{P} \, \SOM{M}^*\, \SOM{S} = \SOM{\Alpha}^*.
    \end{equation}
    The restriction of~\eqref{conjecture:krasner_factorisation} to words without letter~$b$, 
    implies that there exists~$P_0 \subseteq P$ and~$S_0 \subseteq S$ such that~$\PAR{P_0, S_0}$ 
    is a Krasner factorization of size~$n$, where~$a^n \in M$.
    If~$\PAR{P_0, S_0}$ does not borders~$\EnsCal{C}_M$ then
    there exists~$\omega \in \Alpha^*$,~$
    a^{i_1 + n k_1}\omega a^{n \ell_1 + j_1}, 
    a^{i_2 + n k_2}\omega a^{n \ell_2 + j_2} \in M^*,
    p_1, p_2 \in P_0$, and~$s_1, s_2 \in S_0$ such that 
    \begin{equation} \label{conjecture:coef_somme}
    	a^{p_1} \PAR{a^n}^{k_2} 
    	a^{i_1 + n k_1}\omega a^{n \ell_1 + j_1} 
    	\PAR{a^n}^{\ell_2} a^{s_1}
    	=
    	a^{p_2} \PAR{a^n}^{k_1}
    	a^{i_2 + n k_2}\omega a^{n \ell_2 + j_2} 
    	\PAR{a^n}^{\ell_1} a^{s_2},
    \end{equation}
    where~$i_1, i_2, j_1, j_2 \in \entiers$.
    Thus the coefficient of~\eqref{conjecture:coef_somme}
    in~$\SOM{P_0} \, \SOM{M}^* \SOM{S_0}$ is greater or equal to~$2$. 
    Which contradicts the factorization conjecture.
	\end{proof}
    
    According to Theorem~3.2 of~\cite{de1999hajos}, a factorization~$(P,Q)$ is of 
    Haj\'{o}s if and only if there exists a Krasner factorization~$(U,V)$ such 
    that~$(U,Q)$ and~$(P,V)$ are factorizations.    
    Our next Theorem shows a equivalent result for compatible sets.
	We will use the following Lemma according to the proof of 
	Theorem~4.13 of~\cite{szabo2009factoring}.
	
	\begin{lemma} \label{Lemma_krasner_factorization}
    Let~$(U,V)$ be a Krasner factorization of size~$n$.
    If~$U$ is an~$m$-periodic set then for any factorization~$(P,V)$ of size~$n$, the set~$P$
    is~$m$-periodic. 
    \end{lemma}
	
    \begin{theorem} \label{Theorem_carac_periodic}
    	A compatible set is of Haj\'{o}s if and only if its stable is bordered 
    	by a Krasner factorization.
    \end{theorem}

    \begin{proof}
    	We prove by recurrence on~$n$ that any compatible set of~$n$-cbc
    	whose stable is bordered by a Krasner factorization is of Haj\'{o}s. 
    	First, note that the unique (non-empty) compatible set of~$1$-cbc is~$\ENS{\ENS{b}}$ 
    	and that it is of Haj\'{o}s.
    	
    	Assume now that any compatible set of~$j$-cbc (where~$j < n$) whose stable is 
    	bordered by 
    	a Krasner factorization 
    	is of Haj\'{o}s.
    	Let~$\EnsCal{E}$ be a compatible set of~$n$-cbc whose stable is bordered by a 
    	Krasner factorization~$(U,V)$.
    	We can assume that~$n = t_1 \cdots t_k$, where~$t_i > 1$ (for~$i \in [1,k]$),
    	and that~$(U,V)$ is equal to~\eqref{Enum_Krasner} (otherwise, we can 
    	consider~$\DUAL{\EnsCal{E}}$ instead of~$\EnsCal{E}$).
    	
    	If~$k$ is even (respectively odd) then~$U$ (resp.~$V$) is~$t_1 \cdots t_{k-1}$-periodic
    	 and for any~$X \in \EnsCal{E}$ and~$\ell \in \entiers$,
    	\begin{equation*}
    		\PAR{\R{\ell}{n}{a^U X}, V} \hspace{10pt} 
    		\PAR{\text{resp. } \PAR{U, \L{\ell}{n}{ X a^V}}}
		\end{equation*}
		is a factorization.
    	Moreover according to 
    	Lemma~\ref{Lemma_krasner_factorization}, we know that
    	\begin{equation*}
    		\R{\ell}{n}{a^U X} \hspace{10pt} 
    		\PAR{\text{resp. } \L{\ell}{n}{ X a^V}}
    	\end{equation*}
    	 is also~$t_1 \cdots t_{k-1}$-periodic.
    	
    	Thus according to 
    	Proposition~\ref{prop:critere_periodique},~$\EnsCal{E} \in \EnsCal{H}_{t_k}\PAR{\EnsCal{E'}}$ 
    	(resp.~$\DUAL{\EnsCal{E}} \in \EnsCal{H}_{t_k}\PAR{\EnsCal{E'}}$), 
    	where~$\EnsCal{E'}$ is 
    	a compatible set of~$t_1 \cdots t_{k-1}$-cbc whose stable is bordered by 
    	the Krasner factorization~$(U,V)$, where~$k$ is decremented (i.e.~$k \leftarrow k-1$).
    	Thanks to the recurrence hypothesis,~$\EnsCal{E'}$ is of Haj\'{o}s thus~$\EnsCal{E}$ is also 
    	of Haj\'{o}s. 
    	
    	The converse is a straight forward recurrence. Indeed,~$\ENS{\ENS{b}}$ is bordered by the 
    	Krasner factorization~$(\ENS{0}, \ENS{0})$ and, according to 
    	Proposition~\ref{proposition_cbc_period}, 
    	if~$\EnsCal{E}$ is a compatible set 
    	of~$n$-cbc whose stable is bordered by a Krasner factorization~$(U,V)$
    	then~$\EnsCal{E}' \in \EnsCal{H}_{t}\PAR{\EnsCal{E}}$
    	(resp.~$\DUAL{\EnsCal{E}'} \in \EnsCal{H}_{t}\PAR{\DUAL{\EnsCal{E}}}$)
    	is bordered by the Krasner 
    	factorization~$(U+n\ENTIERS{t},V)$ (resp.~$(U, V+n\ENTIERS{t})$).
    \end{proof}

    According to Theorem~\ref{Theorem_carac_periodic} and its constructive proof, we know that given a 
    stable set~$\EnsCal{S}$ (such as~$\EnsCal{C}_M$, when~$M$ is a 
    code that satisfies the factorization conjecture) 
    bordered by a Krasner factorization~$(U,V)$ (we can suppose that it is 
    equal to~\eqref{Enum_Krasner} and that~$k$ is even, the others cases are similar), we have
	\begin{equation*}
		Y \in H_{t_k}\PAR{Y_{k-1}},\,
		\DUAL{Y_{k-1}} \in H_{t_{k-1}}\PAR{Y_{k-2}},\,
		\dots,
		\DUAL{Y_{2}}    \in H_{t_2}\PAR{Y_1},\,
		\DUAL{Y_1} \in H_{t_1}\PAR{\ENS{b}},
	\end{equation*}
    for all~$Y \in \EnsCal{S}$, where~$Y_i = \MODn{Y}{t_{1}\cdots t_{i}}$.
    
    \section{Cbc Haj\'{o}s numbers}
    
	In this section, we fully characterize \DEF{cbc Haj\'{o}s numbers}. They are numbers~$n$
	such that every compatible sets of~$n$-cbc are of Haj\'{o}s. 
	This is sum up in Theorem~\ref{cbc_hajos_numbers:theorem}.    
    
    \subsection{Haj\'{o}s cases}
    
	It is well known in theory of factorizations of abelian groups that given a 
	factorization~$\PAR{P,Q}$ such that~$\CARD{P}$ is a power of a prime then either~$P$ 
	or~$Q$ is periodic. See for example Theorem~6.1.1 from~\cite{szabo2009factoring}. 
	Inspired by Proposition~3.1 from~\cite{SzaboCyclicGroups}, we prove the 
	following slightly stronger result for the particular case of cyclic groups. 
	
    \begin{proposition} \label{hajos_cases:prop:facto}
    	If~$\PAR{P,Q_1}$ and~$\PAR{P,Q_2}$ are factorizations of size~$n$ 
    	such that~$\CARD{P}$ is a power of a 
    	prime then either~$P$ is periodic or~$Q_1$ and~$Q_2$ share a common period.
    \end{proposition}
    
    In order to prove it, we will use the following lemma
    stated as Theorem~5.5 in~\cite{szabo2009factoring}.
	\begin{lemma} \label{hajos_cases:lemma:engendre}
		If~$\PAR{P,Q}$ is a normalized factorization of size~$n$ and~$\CARD{P} = p^\alpha q^\beta$,
		where~$p,q$ are primes and~$\alpha, \beta \geq 0$,
		then either~$\ENGENDRE{P}$ (the subgroup generated by~$P$) is not equal to $\Z$
		or~$\ENGENDRE{Q} \neq \Z$.
	\end{lemma}
    
    \begin{proof}[Proof of Proposition~\ref{hajos_cases:prop:facto}]
    	We prove it by a recurrence on~$n$.
    	We can suppose that~$\PAR{P,Q_1}$ and~$\PAR{P,Q_2}$ are normalized factorizations of 
    	size~$n$ and that~$\CARD{P} = p^\alpha$, where~$p$ is prime and~$\alpha \geq 0$.
    	
    	If~$\CARD{P} = 1$ then~$P = \ENS{0}$ and~$Q_1 = Q_2 = \entiers$ (in~$\Z$) and thus 
    	they verify the
    	proposition. Symmetrically, if~$\CARD{Q_1} = \CARD{Q_2} = 1$ 
    	then~$P = \entiers$ (in~$\Z$) and~$Q_1 = Q_2 = \ENS{0}$ and thus the proposition is satisfied.
    	
    	Suppose that the proposition is true for every factorizations of size~$k < n$.
    	According to Lemma~\ref{hajos_cases:lemma:engendre} 
    	either~$\ENGENDRE{P} \neq \Z$ or~$\ENGENDRE{Q_1} \neq \Z$ and~$\ENGENDRE{Q_2} \neq \Z$.
    	In the first case, there exists a prime~$t \, {\mid}\, n$ such that~$P \subseteq t \Zn{}$. 
    	According to Lemma~2.4 from~\cite{szabo2009factoring}, 
    	for all~$q_1 \in Q_1$ and~$q_2 \in Q_2$,
    	\begin{equation} \label{hajos_cases:eq:facto_1}
    		\PAR{\frac{1}{t}P, \frac{1}{t}\PAR{\PAR{Q_i-q_i} \cap t\Zn{}}} 
    		\hspace{15pt} (\text{where } i = 1,2)
		\end{equation}
    	are normalized factorizations of size~$\frac{n}{t}$.
    	
    	By recurrence hypothesis, either~$\frac{1}{t}P$ is periodic in~$\Zn{\frac{n}{t}}$ and 
    	thus~$P$ is 
    	periodic in~$\Z$ or the right sides of~\eqref{hajos_cases:eq:facto_1} share a common 
    	period~$g$ in~$\Zn{\frac{n}{t}}$.
    	For the second case, we have that
		\begin{equation*}
			tg \in \bigcap\limits_{q_j \in Q_i}\PAR{Q_i - q_j},
		\end{equation*}		
		for~$i = 1,2$. Thus~$tg$ is a common period of~$Q_1$ and~$Q_2$ in~$\Z$ 
		according to Lemma~2.8 from~\cite{szabo2009factoring}.
		
		Last case occurs when~$\ENGENDRE{P} = \Z$ and thus~$\ENGENDRE{Q_1} \neq \Z$.
		There exists a prime~$t \, {\mid}\, n$ such that~$Q_1 \subseteq t \Zn{}$. 
		If~$t \neq p$ 
		then~$tP + Q_1 \subseteq t \Zn{} \neq \Z$ which contradicts Proposition~3 from~\cite{Sands}. 
		Thus~$t=p$ and~$Q_1 \subseteq p \Zn{}$. We also get~$Q_2 \subseteq p \Zn{}$ 
		with the same argument. 
		
		As similar as before, for all~$p_j \in P$,
		\begin{equation} \label{hajos_cases:eq:facto_2}
    		\PAR{\frac{1}{p}\PAR{\PAR{P-p_j} \cap p\Zn{}}, \frac{1}{p}Q_i} 
    		\hspace{15pt} (\text{where } i = 1,2)
		\end{equation}
		are normalized factorizations of size~$\frac{n}{p}$.
		By recurrence hypothesis, either~$\frac{1}{p}Q_1$ and~$\frac{1}{p}Q_2$ share a commune 
		period 
		in~$\Zn{\frac{n}{p}}$, and thus~$Q_1$ and~$Q_2$ share a commune 
		period in~$\Z$, or 
		 the left sides of~\eqref{hajos_cases:eq:facto_2} are periodic in~$\Zn{\frac{n}{p}}$.
		 
		 For the second case, since their cardinalities are powers of~$p$ then
		 they share a common period~$g$ in~$\Zn{\frac{n}{p}}$ (take~$g$ as the maximum of their 
		 periods, for example).
    	Thus, we have that
		\begin{equation*}
			pg \in \bigcap\limits_{p_j \in P}\PAR{P - p_j}.
		\end{equation*}		
		So~$pg$ is a period of~$P$ in~$\Z$, according to Lemma~2.8 from~\cite{szabo2009factoring}. 
		This concludes the proof.
	\end{proof}    	    
    
    We extend Proposition~\ref{hajos_cases:prop:facto} to compatible sets. 
	
	\begin{theorem} \label{th:bord_premier}
    If~$\EnsCal{E}$ is a compatible set of cbc such that its stable is bordered by~$\PAR{P,Q}$ 
    where~$\CARD{P}$ is a power of a 
    prime then~$\EnsCal{E}$ is of Haj\'{o}s.
    \end{theorem} 
    
    \begin{proof}
    We prove it by recurrence.
    	Let~$\EnsCal{E}$ be a compatible set of~$n$-cbc whose stable is bordered by~$\PAR{P,Q}$.
    	If~$\CARD{P} = 1$ (resp.~$\CARD{Q} = 1$) then the Krasner 
    	factorization~$\PAR{\ENS{0}, \entiers}$ (resp.~$\PAR{\entiers, \ENS{0}}$)
    	borders~$\EnsCal{E}^\circ$ and thus~$\EnsCal{E}$ is of Haj\'{o}s according 
    	to Theorem~\ref{Theorem_carac_periodic}.
    	
    	Suppose that the proposition is true for every compatible set of~$i$-cbc, where~$i < n$.
    	Let
    	\begin{equation*}
    		\EnsCal{L} := \ENS{Q} \cup 
    		\ENS{\L{k}{n}{X a^Q}, k \in \entiers, X \in \EnsCal{E}}
    	\end{equation*}
    	and 
    	\begin{equation*}
    		\EnsCal{R} := \ENS{P} \cup 
    		\ENS{\R{k}{n}{a^P X}, k \in \entiers, X \in \EnsCal{E}}.
    	\end{equation*}
    	
    	For any~$L \in \EnsCal{L}, R \in \EnsCal{R}$, the ordered pair~$\PAR{R,L}$ 
    	is a factorization where~$\CARD{R}$ is a power of a prime thus, 
    	according to Proposition~\ref{hajos_cases:prop:facto},    	
    	either elements of~$\EnsCal{L}$ or elements of~$\EnsCal{R}$ share a 
    	common period.
    	In first case (resp. second),
    	according to 
    	Proposition~\ref{prop:critere_periodique},
    	there exists a compatible set~$\EnsCal{E}'$ and~$t > 1$ such 
    	that~$\DUAL{\EnsCal{E}} \in \EnsCal{H}_{t}\PAR{\EnsCal{E'}}$ 
    	(resp.~$\EnsCal{E} \in \EnsCal{H}_{t}\PAR{\EnsCal{E'}}$).
    	Thanks to the recurrence hypothesis,~$\EnsCal{E'}$ is of Haj\'{o}s 
    	thus~$\EnsCal{E}$ is also 
    	of Haj\'{o}s according to Proposition~\ref{proposition_cbc_period}. 
    \end{proof}

	Theorem~\ref{th:bord_premier} provides two straight forward corollaries that characterize 
	cbc Haj\'{o}s numbers.
    
    \begin{corollary} \label{coro:trois_premiers}
   	If~$n$ is the product of at most three primes 
   	(eventually equal) then it is a cbc Haj\'{o}s number.
    \end{corollary}   
    
    \begin{proof}
    Let~$\EnsCal{E}$ be a compatible set of~$n$-cbc, where~$n$ is the product of at most three primes.
    According to Theorem~\ref{th:bord:FN},~$\EnsCal{E}^\circ$ is bordered by a 
    factorization~$\PAR{P,Q}$.
    Since~$n = \CARD{P}\times\CARD{Q}$, either~$\CARD{P}$ or~$\CARD{Q}$ is equal to~$1$ or 
    a prime thus~$\EnsCal{E}$ is of Haj\'{o}s according to Theorem~\ref{th:bord_premier}. 
    This concludes the proof.
    \end{proof}
    
    \begin{corollary} \label{coro:p_k_q}
   	Numbers of the form~$p^kq$, where~$k \geq 0$ and~$p,q$ are primes, are cbc Haj\'{o}s numbers.
    \end{corollary}  
    
    \begin{proof}
    Let~$\EnsCal{E}$ be a compatible set of~$p^kq$-cbc, where~$k \geq 0$ and~$p,q$ are primes.
    According to Theorem~\ref{th:bord:FN},~$\EnsCal{E}^\circ$ is bordered by a 
    factorization~$\PAR{P,Q}$.
    Since~$p^kq = \CARD{P}\times\CARD{Q}$, either~$\CARD{P}$ or~$\CARD{Q}$ is a power of~$p$
     thus~$\EnsCal{E}$ is of Haj\'{o}s according to Theorem~\ref{th:bord_premier}. 
    This concludes the proof.
    \end{proof}
    
	Haj\'{o}s characterization provides some simple enumerative formulas. 	    
    
    \begin{example}
    	Given a prime number~$p$, we can enumerate and count~$p$-cbc which are also 
    	Haj\'{o}s cbc of size~$p$, according to Corollary~\ref{coro:trois_premiers}. 
    	We have that
    	\begin{equation*}
    	\CARD{H_p} = \CARD{H_p\PAR{\ENS{b}}} + \CARD{\DUAL{H_p\PAR{\ENS{b}}}} - 
    	\CARD{H_p\PAR{\ENS{b}} \cap \DUAL{H_p\PAR{\ENS{b}}}}.
    	\end{equation*}
    	We first enumerate the set~$H_p\PAR{\ENS{b}}$ which is equal to
		\begin{equation*}
    		\ENS{
   			\ENS{a^{k_{1}}b,a^{k_{2}}ba, \dots, 
   			a^{k_{p}}ba^{p-1}},
   			k_{1}, \dots, k_{p} \in \ENTIERS{p}}.
    	\end{equation*}    	
    	Thus $\CARD{H_p\PAR{\ENS{b}}} = p^p$. Similarly, we 
    	have~$\CARD{\DUAL{H_p\PAR{\ENS{b}}}} = p^p$.   
    	Moreover, their meet is equal to
    	\begin{equation} \label{exemple:cardinal_meet}
    		\ENS{
   			\ENS{a^{0}ba^{\sigma_1-1},a^{1}ba^{\sigma_2-1}, \dots, 
   			a^{p-1}ba^{\sigma_{p}-1}},
   			\sigma \in \mathfrak{S}_p},
    	\end{equation}
    	where~$\mathfrak{S}_p$ is the group of permutations of size~$p$.
    	Thus the cardinal of~\eqref{exemple:cardinal_meet} is equal to~$p!$.
    	Finally, we have the formula 
    	\begin{equation*}
    		\CARD{H_p} = 2p^p - p!.
    	\end{equation*}
    \end{example}

    \subsection{Non-Haj\'{o}s cases}
    
	In this section, we prove that numbers not concerned by 
	Corollaries~\ref{coro:trois_premiers} and~\ref{coro:p_k_q} are not cbc Haj\'{o}s numbers.
	
	\begin{proposition} \label{non_hajos_non_cbc_hajos}
		Non-Haj\'{o}s numbers are non-cbc Haj\'{o}s numbers.
	\end{proposition}  
	
	\begin{proof}
		Let~$n$ be a non-Haj\'{o}s number and~$\PAR{P,Q}$ be a
    non-Haj\'{o}s factorization of size~$n$.
		Suppose that the~$n$-cbc~$a^Pba^Q$ is of Haj\'{o}s then it is bordered by a 
		Krasner factorization~$\PAR{U,V}$,
		according to Theorem~\ref{Theorem_carac_periodic}. 
		The ordered pairs~$(U,P)$ and~$(Q,V)$ must be factorizations and thus, 
    according to Theorem~3.2 of~\cite{de1999hajos},~$\PAR{P,Q}$ 
    must be a Haj\'{o}s factorization which is a contradiction.
    Thus~$n$ is a non-cbc Haj\'{o}s number.
	\end{proof}	   
    
	Even if the numbers~$p_1^2q_1^2, p_1p_2q_1^2$, and~$p_1p_2q_1q_2$ 
	(when~$p_1,p_2,q_1,q_2$ are distinct primes) 
	are of Haj\'{o}s, we prove in this section that there are not cbc Haj\'{o}s numbers.
	
	We set for the rest of this section, the ordered pairs~$(L,R_1)$ and~$(L,R_2)$, where	
	\begin{equation*}
		L := p_1p_2\ENTIERS{q_1} + q_1q_2\ENTIERS{p_1},
		R_1 := p_1 p_2 q_1\ENTIERS{q_2} + p_1\ENTIERS{p_2},
		R_2 := p_1 q_1 q_2\ENTIERS{p_2} + q_1\ENTIERS{q_2},
	\end{equation*}	
	and~$p_1,p_2,q_1,q_2$ are primes such that~$p_1 p_2 \wedge q_1 = q_1 q_2 \wedge p_1 = 1$.	
	For example, if~$p_1 = 2, p_2 = 2, q_1 = 3,$ and~$q_2 = 3$ then
	\begin{equation*}
		\begin{array}{ccccc}
			L &=& \ENS{0,4,8} + \ENS{0,9} &=& \ENS{0,4,8,9,13,17},\\
			R_1 &=& \ENS{0,12,24} + \ENS{0,2} &=& \ENS{0,2,12,14,24,26},\\
			R_2 &=& \ENS{0,18} + \ENS{0,3,6} &=& \ENS{0,3,6,18,21,24}.
		\end{array}
	\end{equation*}
	
	First, we prove that those ordered pairs are factorizations of size~$n := p_1 p_2 q_1 q_2$.
	
	\begin{proposition} \label{non_Hajos:prop:CE_facto}
	The ordered pairs~$(L,R_1)$ and~$(L,R_2)$ are factorizations.
	\end{proposition}
	
	\begin{proof}
	The sum~$L + R_1$ is equal to
	\begin{equation*}
	p_1\ENTIERS{p_2} + 
	p_1 p_2\ENTIERS{q_1} + 
	p_1 p_2 q_1\ENTIERS{q_2} + 
	q_1 q_2\ENTIERS{p_1}
	=
	p_1 \ENTIERS{p_2q_1q_2} + 
	q_1 q_2 \ENTIERS{p_1}
	\end{equation*}
	Since~$q_1 q_2 \wedge p_1 = 1$ then~$q_1 q_2 \ENTIERS{p_1}$
	is equal to~$\ENTIERS{p_1}$ in~$Z_{p_1}$.
	So~$L + R_1$ is equal to 
	\begin{equation*}
		p_1 \ENTIERS{p_2q_1q_2} + \ENTIERS{p_1}
		= \ENTIERS{n}
	\end{equation*}
	in~$Z_n$.
	Thus~$(L,R_1)$ is a factorization.
	
	Similar argument can be applied to~$(L,R_2)$.
	\end{proof}
	
	Now, we study their periodicity.
	
	\begin{proposition} \label{non_hajos:prop:period_commune_R}
	The sets~$R_1$ and~$R_2$ are periodic in~$Z_n$ without common period 
	and~$L$ is not periodic in~$Z_n$. 
	\end{proposition}
	
	\begin{proof}
	By definition,~$R_1$ and~$R_2$ are periodic in~$\Z$ with respectively period~$p_1 p_2 q_1$ 
	and~$p_1 q_1 q_2$. Since~$\CARD{R_1}=\CARD{R_2}=p_2 q_2$, if~$R_1$ and~$R_2$ share a common 
	period then either~$R_1$ or~$R_2$ has period~$p_1q_1$. 
	Suppose that~$p_1q_1$ is a period of~$R_1$ then~$R_1 = p_1q_1 \ENTIERS{p_2q_2}$ in~$\Z$, which is 
	impossible since~$p_1 = 0 + p_1 \in R_1$ and~$p_1 \not\in p_1q_1 \ENTIERS{p_2q_2} = R_1$.
	Similarly, we prove that~$p_1q_1$ is not a period of~$R_2$. 
	Thus~$R_1$ and~$R_2$ are periodic in~$Z_n$ without common period.
	
	Suppose that~$L$ is periodic in~$\Z$ with period~$g$. 
	Since~$\CARD{L} = p_1q_1$,~$g \in \ENS{p_2q_2, p_2q_2p_1, p_2q_2q_1}$.
	If~$g = p_2q_2$ then~$L = p_2q_2\ENTIERS{p_1q_1}$ in~$\Z$ but~$q_1q_2 \in L$ 
	and~$q_1q_2 \not\in p_2q_2\ENTIERS{p_1q_1}$ thus~$g \neq p_2q_2$.
	Moreover, since for all~$k\in\ENTIERS{p_1}$,~$p_1\wedge q_2q_1k =1$ 
	then there is no~$L'$ such that~$L = L' + p_2q_2p_1\ENTIERS{q_1}$
	and thus~$g \neq p_2q_2p_1$.
	Similarly, we prove that~$g \neq p_2q_2q_1$.
	
	This concludes the proof.	
	\end{proof}
	
	We build a cbc over theses factorizations.	
	Let~$Y$ be a cbc
	\begin{equation*}
		\sum\limits_{\ell \, \in \, L} a^{\ell} b a^{D_\ell},
	\end{equation*}
	where~$D_{(l \in L)} \in \ENS{R_1, R_2}$ and where~$\ell_1, \ell_2 \in L$ 
	be such that~$D_{\ell_1} = R_1$ and~$D_{\ell_2} = R_2$.

	\begin{proposition} \label{non_Hajos:prop:CE}
		The set~$Y$ is a non-Haj\'{o}s cbc.
	\end{proposition}	
	
	\begin{proof}
		Suppose that~$Y$ is a Haj\'{o}s cbc then it is bordered by a Krasner factorization~$\PAR{U,V}$.
		In particular,~$\PAR{U,L}, \PAR{R_1, V},$ and~$\PAR{R_2, V}$ must be factorizations.
		According to Theorem~3.2 of~\cite{de1999hajos},~$\PAR{U,L}$ is of Haj\'{o}s and 
		since~$L$ is not periodic then~$U$ is periodic.
		Moreover, according to Lemma~\ref{Lemma_krasner_factorization},
		the sets~$U,R_1,$ and~$R_2$ 
		must share a common period which is contradicted by 
		Proposition~\ref{non_hajos:prop:period_commune_R}.		 
	\end{proof}
	
	Proposition~\ref{non_Hajos:prop:CE} implies that 
	numbers of the form~$p_1^2q_1^2, p_1p_2q_1^2$, and~$p_1p_2q_1q_2$,
	where~$p_1,p_2,q_1,q_2$ are distinct primes, 
	are of non-cbc Haj\'{o}s numbers.

	\begin{example}
	According to Proposition~\ref{non_Hajos:prop:CE}, the $36$-cbc 
	\begin{equation*}
	\ENS{a^{36}} \cup ba^{\ENS{0,2,12,14,24,26}} \cup 
	a^{\ENS{4,8,9,13,17}}ba^{\ENS{0,3,6,18,21,24}} 
	\end{equation*}
	is not of Haj\'{o}s. Similarly to Proposition~\ref{non_Hajos:prop:CE}, we can show that the code
	\begin{equation*}
	\ENS{a^{36}} \cup 
	a^{\ENS{0,4,8,9,13,17}}ba^{\ENS{0,2,12,14,24,26}} \cup 
	a^{\ENS{0,4,8,9,13,17}}ca^{\ENS{0,3,6,18,21,24}} 
	\end{equation*}
	over the alphabet~$\ENS{a,b,c}$ is not of Haj\'{o}s.
	If one of them is included in a finite maximal code then it would not be 
	bordered by a Krasner factorization
	and thus it would provide a counterexample to
	the factorization conjecture.
	\end{example}

	Thanks to Corollaries~\ref{coro:trois_premiers} and~\ref{coro:p_k_q} and
	Propositions~\ref{non_hajos_non_cbc_hajos} and~\ref{non_Hajos:prop:CE}, we conclude 
	this section 
	by providing the exhaustive list of cbc Haj\'{o}s numbers.
	
	\begin{theorem} \label{cbc_hajos_numbers:theorem}
	Cbc Haj\'{o}s numbers are product of at most three primes or numbers of the form~$p^kq$, 
	where~$k \geq 0$ and~$p,q$ are primes. 
	\end{theorem}
	
   	Smallest non-cbc Haj\'{o}s numbers 
   	are therefore~$2^2 3^2 = 36$,~$2^2 3 \times 5 = 60$, and~$2^3 3^2 = 72$. 
   	It is referred as sequence~A320632 in~\cite{oeis2}.
	
    \section{Prefix-suffix codes}
    
    We recall that given two codes~$C_1$ and~$C_2$ over the 
    alphabet~$\Alpha$, the code~$C_1$ is said to be a \DEF{composition} of~$C_2$
	if and only if~$C_1$ is a code over the alphabet~$C_2$ (i.e.~$C_1 \subseteq {C_2}^*$).
	For example, the code
	\begin{equation} \label{exemple_code_prefix-suffix}
		\ENS{aa, ab, abbab, bbaa}
	\end{equation}
	is a composition of the 
	code~$\ENS{aa,ab,b}$ since it is equal 
	to
	\begin{equation*}
		\ENS{aa, ab, \PAR{ab}\PAR{b}\PAR{ab}, \PAR{b}\PAR{b}\PAR{aa}}.
	\end{equation*}
	Of course, a code is always a composition of himself and of its alphabet.
	
	We recall that according to Proposition~2.6.1 from~\cite{berstel2010codes},
	if~$C_2$ is a code (over~$\Alpha$) and~$C_1$ is a code over~$C_2$ 
	then~$C_1$ is a code over~$\Alpha$.  
	A code~$C_1$ is recursively said to be a \DEF{prefix-suffix} code over~$C_2$ 
	if it is equal to~$C_2$
	or if it is a prefix or suffix code over a \textit{prefix-suffix} code over~$C_2$.	
	For example, the code~\eqref{exemple_code_prefix-suffix} is a prefix code 
	over
	\begin{equation*}
		\ENS{ aa, ab, abb, bb }
	\end{equation*}		
	which is a suffix code
	over
	\begin{equation*}
		\ENS{ aa, ab, b }
	\end{equation*}		
	which is again a prefix code over~$\Alpha$.
	 Thus~\eqref{exemple_code_prefix-suffix} is a 
	prefix-suffix code (over~$\Alpha$).
	We simply say \textit{prefix-suffix code} when it is a prefix-suffix code over~$\Alpha$.
	Prefix-suffix codes are included in finite maximal codes according 
	to Corollary~$1$ from~\cite{restivo1989completing}.    
    
    We have the following proposition, inspired by Lemmas~$3.3$ and~$3.4$ from~\cite{Lam}.
	
	\begin{lemma} \label{completion:prop:hajos}
	If~$C$ is a code containing~$a^n$ then	
	for any~$i_{1}(\omega), \dots, i_{t}(\omega),j_{1}(\omega), \dots, j_{t}(\omega) \geq 0$, 
	where~$\omega \in C\setminus \ENS{a^n}$, the set
	\begin{equation} \label{completion:prop:composition}
		\ENS{a^{nt}} \cup 
		\bigsqcup\limits_{\omega \in C\setminus a^n} 
		\ENS{
		a^{n i_{1}(\omega)} \omega a^{n t j_{1}(\omega)},
		a^{n i_{2}(\omega)} \omega a^{n (1+t j_{2}(\omega))} , \dots, 
		a^{n i_{t}(\omega)} \omega a^{n (t-1+t j_{t}(\omega))}} 
	\end{equation}		
	is a prefix-suffix code over~$C$.
	\end{lemma}	
	
	\begin{proof}
		The set
		\begin{equation} \label{completion:prop:composition_1}
		\ENS{\PAR{a^{n}}^t} \cup 
		\bigsqcup\limits_{\omega \in C\setminus a^n} 
		\ENS{\PAR{a^{n}}^{i_{1}(\omega)} \omega,
		\PAR{a^{n}}^{i_{2}(\omega)} \omega \PAR{a^{n}}, \dots, 
		\PAR{a^{n}}^{i_{t}(\omega)} \omega \PAR{a^{n}}^{t-1}}
		\end{equation}
		is a suffix code over the code~$C$ and
		\begin{equation*} 
		\ENS{a^{nt}} \cup 
		\bigsqcup\limits_{\omega \in C \setminus a^n} 
		\ENS{\PAR{a^{n i_{1}(\omega)} \omega} \PAR{a^{nt}}^{j_{1}(\omega)},
		\PAR{a^{n i_{2}(\omega)} \omega a^n} \PAR{a^{n t}}^{j_{2}(\omega)}, \dots, 
		\PAR{a^{n i_{t}(\omega)} \omega a^{n (t-1)}} \PAR{a^{nt}}^{j_{t}(\omega)}}
		\end{equation*}
		is a prefix code over the code~\eqref{completion:prop:composition_1}.
		Thus the code~\eqref{completion:prop:composition} is 
		a prefix-suffix code over~$C$.
	\end{proof}

	We deduce from this lemma a theorem about completion.	
	
	\begin{theorem} \label{completion:th1}
		Let~$\EnsCal{E}$ be a Haj\'{o}s 
		compatible set of~$n$-cbc, where~$n > 1$, and~$C := \ENS{a, \omega_1, \dots, \omega_k}$ 
		a prefix-suffix code. 
		For any~$X_1, \dots, X_k \subseteq a^*ba^*$		
		such that~$\MOD{X_1}, \dots, \MOD{X_k} \in \EnsCal{E}$, the set
		\begin{equation} \label{completion:th:eq:alphabet}
			\ENS{a^n} \cup \bigsqcup\limits_{i = 1}^{k} 
			X_i[b \leftarrow \omega_i],
		\end{equation}	
		where~$X[b \leftarrow \omega]$ is the set of words~$X$ whose letters~$b$ 
		are replaced by word~$\omega$,	
		is a prefix-suffix code and thus it is		
		included in a finite maximal code. 
	\end{theorem}
	
	\begin{proof}
	We prove it by a recurrence on~$n$.
	
	Since~$\EnsCal{E}$ is of Haj\'{o}s then there exists
	a compatible set~$\EnsCal{E'}$ of~$m$-cbc such that~$n=mt, t > 1,$ 
	and~$\EnsCal{E} \in \EnsCal{H}_{t}\PAR{\EnsCal{E}'}$ 
	or~$\DUAL{\EnsCal{E}} \in \EnsCal{H}_{t}\PAR{\EnsCal{E}'}$.
	We can assume that~$\EnsCal{E} \in \EnsCal{H}_{t}\PAR{\EnsCal{E}'}$, the other case is similar.
	Therefore, for all~$X_1, \dots, X_k \in \EnsCal{E}$, 
	there exists~$Y_1, \dots, Y_k \in \EnsCal{E}'$ such that~$X_i \in H_{t}\PAR{Y_i}$ 
	for all~$i \in [1,k]$.
	
	If~$m=1$ then~$X_i \in H_{t}\PAR{\ENS{b}}$, for all~$i \in [1,k]$. 
	Moreover, since~$C$ is a prefix-suffix code then
	according to Proposition~\ref{completion:prop:hajos},
	the set~\eqref{completion:th:eq:alphabet} is also a 
	prefix-suffix code.	
	
	Otherwise (when~$m>1$) we can assume, by recurrence hypothesis, that
	\begin{equation*} 
			\ENS{a^m} \cup \bigsqcup\limits_{i = 1}^{k} 
			Y_i[b \leftarrow \omega_i]
	\end{equation*}	
	is a prefix-suffix code and thus 
	according to Proposition~\ref{completion:prop:hajos},
	the set~\eqref{completion:th:eq:alphabet} is also a 
	prefix-suffix code.	
	\end{proof}
	
	Theorem~$3.2$ from~\cite{Lam} is the particular case of Theorem~\ref{completion:th1} 
	where~$C = \ENS{a,b}$ and~$\EnsCal{E}$ is made of one cbc of the form~$a^Pba^Q$.
	Note that our alphabet~$\Alpha$ does not have to be binary.
	
	Our next corollary is a small step towards the inclusion problem.	
	
	\begin{corollary}
		Let~$n$ be a cbc Haj\'{o}s number,~$\omega \in \Alpha^*\setminus a^*$, 
		and~$X \subseteq a^* \omega a^*$.
		Considering the set~$\ENS{a^n} \cup X$, the following statements are equivalent:
		\begin{enumerate}
			\item it is included in a finite maximal code,
			\item $C_X\PAR{\omega}$ is included in an~$n$-cbc,
			\item $C_X\PAR{\omega}$ is included in an~$n$-Haj\'{o}s cbc,
			\item it is a prefix-suffix code.
		\end{enumerate}
	\end{corollary}		
	
	\begin{proof}
		Statement~$1$ implies Statement~$2$ according to the recalls made in 
		Section~$1$ and
		Statement~$2$ implies Statement~$3$ since~$n$ is a cbc Haj\'{o}s number.
		
		Suppose that Statement~$3$ is true. 
		Let~$Y$ be a Haj\'{o}s cbc that contains~$C_X\PAR{\omega}$.
		The set~$\ENS{a, \omega}$ is a code since~$\omega \in \Alpha^*\setminus a^*$ 
		and it is prefix-suffix
		because any code with two elements is prefix-suffix according to 
		Theorem~$3$ from~\cite{restivo1989completing}. 
		Thus according to 
		Theorem~\ref{completion:th1},
		\begin{equation*}
			\ENS{a^n} \cup X \subseteq 
		\ENS{a^n} \cup Y[b \leftarrow \omega]
		\end{equation*}
		is a prefix-suffix code.	
		This proves that Statement~$3$ implies Statement~$4$.
		
		We recall that Statement~$4$ implies Statement~$1$ according 
		to Corollary~$1$ from~\cite{restivo1989completing}.
	\end{proof}
	
	For example, the code
	\begin{equation}
		\ENS{ aaab, aaba, b, ba }
	\end{equation}
	is not prefix-suffix (we do not prove it, we just did a computer check). 
	Thus if it is included in a finite maximal code 
	then it would not be bordered by a Krasner factorization and thus it would
	 provides a counterexample to the factorization conjecture.
	Such a code would contain a word of the form $a^n$ where $n$ is a non-cbc Haj\'{o}s number, 
	in particular~$n \geq 36$.

	\begin{remark}
		There is a converse to Theorem~\ref{completion:th1}. Indeed, any prefix-suffix code
		is included in a prefix-suffix finite maximal code. Such a code, let call it~$M$,
		satisfies the factorization conjecture, according to Proposition~14.1.2
		 from~\cite{berstel2010codes}. 
		Thus according to Proposition~\ref{prop:conj_facto_bord_krasner},~$C_M$ is bordered 
		by a Krasner factorization and thus it is of Haj\'{o}s according to
		 Theorem~\ref{Theorem_carac_periodic}.
	\end{remark}

	Next Theorem provides the best known bound for (the strong version of) the 
	long-standing \textit{triangle conjecture}. 
	
    \begin{theorem} \label{borne_conj_triangle}
    	The strong triangle conjecture is true for the particular cases 
    	where~$n$ is a cbc Haj\'{o}s number.
    \end{theorem}
    
    \begin{proof}
    	If~$X$ is an~$n$-cbc where~$n$ is a cbc Haj\'{o}s number then it is prefix-suffix
    	according to Theorem~\ref{completion:th1}. 
    	Moreover, according to Example~14.6.1 and Proposition~14.6.3 from~\cite{berstel2010codes}, 
    	any prefix-suffix cbc verifies the triangle property.
    \end{proof}
    
    According to Theorem~\ref{borne_conj_triangle}, 
    the strong triangle conjecture (and thus the Zhang and Shum conjecture) 
    is, in particular, true when~$n < 36$.

	\section{Not commutatively prefix bayonet codes}
	
	In this section, we prove a conjecture about the size 
	of a potential
	counterexample to the triangle conjecture.
	
	A list of codes 
	that do not verify the original triangle 
	conjecture\footnote{the first counterexample was found by Shor, 
	as recalled in the introduction.} is exhibit in~\cite{cordero2019note, cordero2019explorations},
	they are called \DEF{not commutatively prefix} bayonet codes.	
	We recall that if one of those is included in a finite maximal code 
	then the triangle conjecture and the factorization conjecture
	are false. And a necessary condition for a bayonet code  to be 
	included in a finite maximal code
	is to be included in a cbc.
	
	The following conjecture about the divisibility of~$n$ such 
	that an~$n$-cbc contains
	a given bayonet code is proposed in~\cite{cordero2019note}. 
	\begin{conjecture} \label{conj:d_conjecture}
		For any~$n$-cbc~$X$ and~$d$ prime to~$n$, the set
		\begin{equation*}
			\varphi_{d}(X) := 
			\ENS{a^i b a^{\MOD{d j}} \text{ such that } a^i b a^{j} \in X}
		\end{equation*}
		is an~$n$-cbc.
	\end{conjecture}
	
	We prove a stronger version of this conjecture.
	
	\begin{theorem} \label{th:d_conjecture}
		Let~$\EnsCal{E}$ be a compatible set of~$n$-cbc such that~$\EnsCal{E}^\circ$ is 
		bordered by~$(P,Q)$.
		For any~$X \in \EnsCal{E}$,~$d_1$ prime to~$\CARD{Q}$, and~$d_2$ prime to~$\CARD{P}$, 
		the set
		\begin{equation*}
			\ENS{\varphi_{d_1,d_2}(X)} \cup \EnsCal{E},
		\end{equation*}
		where
		\begin{equation*}
			\varphi_{d_1,d_2}(X) := \ENS{a^{\MOD{d_1 i}} b a^{\MOD{d_2 j}}
			\text{ such that } a^i b a^{j} \in X},
		\end{equation*}
		is a compatible set of~$n$-cbc and its stable is bordered by~$(P,Q)$. 
	\end{theorem}	
	
	\begin{proof}
	We prove by a recurrence on~$k$ the following property:
	for all~$X_1, \dots, X_j \in \EnsCal{E} \cup \ENS{\varphi_{d}(X)}$ such that
	\begin{equation*}
		\CARD{\ENS{X_i \text{ such that } i \in [1,j] \text{ and } X_i = \varphi_{d}(X)}} \leq k,
	\end{equation*}
	we have
	\begin{equation*} 
			a^{P} \SOM{X_1} \cdots \SOM{X_j} a^{Q}
   			\Modn{n} 
    		\PAR{a^{\ENTIERS{n}}b}^j a^{\ENTIERS{n}}.
   	\end{equation*}
   	
   	It is true for~$k = 0$ because~$\EnsCal{E}$ is a compatible set whose stable 
   	is bordered by~$(P,Q)$. 
   	Suppose now that it is true for~$k$. 
   	Let~$X_1, \dots, X_j \in \EnsCal{E} \cup \ENS{\varphi_{d}(X)}$ be such that
	\begin{equation*}
		\CARD{\ENS{X_i \text{ such that } i \in [1,j] 
		\text{ and } X_i = \varphi_{d}(X)}} = k+1
	\end{equation*}
	and let~$\ell$ be
	\begin{equation*}
		\text{min}\ENS{i \text{ such that } X_i = \varphi_{d}(X)}.
	\end{equation*}
	By recurrence hypothesis, 
	\begin{equation*} 
			a^{P} \SOM{X_1} \cdots \SOM{X_{\ell-1}} \, \SOM{X}
			\, \SOM{X_{\ell+1}} \cdots \SOM{X_{j}} a^{Q}
   			\Modn{n} 
    		\PAR{a^{\ENTIERS{n}}b}^j a^{\ENTIERS{n}}
   	\end{equation*}
   	thus for all~$i, i_1, \dots, i_{\ell-1} \in \entiers$,
   	\begin{equation*} 
   			\PAR{ \R{i}{n}{a^{P} X_1 \circ_{i_1} \cdots  X_{\ell-1} \circ_{i_{\ell -1}} X}
   			, Q}
   	\end{equation*}
   	are borders of~$\EnsCal{E}^\circ$.
   	According to Proposition~\ref{prop:bord_trans}, 
   	for all~$i, i_1, \dots, i_{\ell-1} \in \entiers$ and~$d$ prime to~$\CARD{P}$,
   	\begin{equation*} 
   			\PAR{ d \R{i}{n}{a^{P} X_1 \circ_{i_1} \cdots  X_{\ell-1} \circ_{i_{\ell -1}} X}
   			, Q}
   	\end{equation*}
   	are also borders of~$\EnsCal{E}^\circ$.
   	So
   	\begin{equation*} 
   		\PAR{\R{i}{n}{a^{P} X_1  \circ_{i_1} \cdots  X_{\ell-1} \circ_{i_{\ell -1}} \varphi_{d}(X)},
   		Q}
   	\end{equation*}
   	are borders of~$\EnsCal{E}^\circ$ and thus
	\begin{equation*} 
			a^{P} \SOM{X_1} \cdots \SOM{X_{j}} a^{Q}
   			\Modn{n} 
    		\PAR{a^{\ENTIERS{n}}b}^j a^{\ENTIERS{n}}.
   	\end{equation*}
   	Thus the proposition is true for~$k+1$.
   	This proves that~$\EnsCal{E} \cup \ENS{\varphi_{d}(X)}$ is a compatible set.
   	We obtain the expected result by duality.   	 
	\end{proof}

	Theorem~\ref{th:d_conjecture} does imply Conjecture~\ref{conj:d_conjecture} 
	because according to Proposition~\ref{prop:stable}, 
	for any~$n$-cbc~$X$, there exists an ordered pair~$(P,Q)$ which borders~$\ENS{X}^\circ$ 
	and any number prime to~$n = |P|\times|Q|$ is also prime to~$|P|$ and~$|Q|$.	
	
	Theorem~\ref{th:d_conjecture} allows us to compute some
	lower bounds about potential counterexamples to the 
	triangle conjecture.  
	
	\begin{example}
	According to~\cite{cordero2019note}, one of the four 
	smallest (for cardinality) not commutatively prefix bayonet codes is
	\begin{equation*}
	T := \ENS{b, ba^2, ba^8, ba^{10}, aba^8,
	 aba^{10},
	a^4b, a^4ba^2, a^5b,  a^5ba^3,
	a^5ba^6, a^9b, a^9ba^2}.
	\end{equation*}
	We already know\footnote{see section~3.1 of~\cite{cordero2019note}}, 
	thanks to computer exploration and factorization theory, 
	that if~$T$ is included in an~$n$-cbc then~$n = 4k$, where~$k \geq 8$.	
	
	We note that 
	\begin{equation*}
		\PAR{ ba^{2\times2} } (b) = (b) \PAR{ a^4b }
		\text{ and }
		\PAR{ a^5ba^{3\times3} } (b) = \PAR{ a^5b } \PAR{ a^9b}
	\end{equation*}	
	thus~$\mu_{1,2}\PAR{T}$
	and~$\mu_{1,3}\PAR{T}$ are not codes, where
	\begin{equation*}
		\mu_{d_1,d_2}(T) := \ENS{a^{d_1 i} b a^{d_2 j}
		\text{ such that } a^i b a^{j} \in T}.
	\end{equation*}
	Likewise, we note that~$\mu_{2,1}\PAR{T}$ and~$\mu_{3,1}\PAR{T}$ 
	are not codes.
	Thus according to Theorem~\ref{th:d_conjecture}, if~$T$ is included in an~$n$-cbc~$X$ then 
	any border~$(P,Q)$ of~$\ENS{X}^\circ$
	is such that~$2\times3\,|\,|P|$ and~$2\times3\,|\,|Q|$, 
	thus~$36 \, |\, n$. 
	
	Similar argument can be applied to others known not commutatively 
	prefix bayonet codes. 
	\end{example}
	
    \section*{Conclusion and perspectives}

    We conclude this article by exposing our main perspectives.
    We do not conjecture that the general case of the strong triangle conjecture is true. 
	We believe that techniques developed in order to build non-Haj\'{o}s factorizations 
	and non-\textit{Rédei} factorizations such as in~\cite{sands2007question} could be useful 
	to create counterexamples to the strong triangle conjecture.	
	Since every counterexample of the (Zhang and Shum) triangle conjecture must 
	contain a counterexample to the 
	strong triangle conjecture, we believe that it is an intermediate step in order to find a
	 counterexample to the triangle conjecture (if it exists).
	
	Our second perspective is the converse of Proposition~\ref{prop:conj_facto_bord_krasner},
	we wounder if every finite maximal codes bordered by Krasner factorizations satisfy the
	 factorization conjecture. Thanks to the characterization provided by
	  Theorem~\ref{Theorem_carac_periodic},
	 we are more confident about a positive answer. 
	 If it is the case then results about the triangle conjecture from
	  Theorem~\ref{borne_conj_triangle} 
	 could be extend to the factorization conjecture.

	\bibliographystyle{alphaurl}
	\bibliography{biblio}
	
\end{document}